\pdfoutput=1
\documentclass{amsart}
\usepackage{amsmath}
\usepackage{amsthm}
\usepackage{amssymb}
\usepackage{graphicx}
\usepackage{url}
\theoremstyle{plain}
\newtheorem{theorem}{Theorem}
\newtheorem{fact}{Fact}
\newtheorem{lemma}{Lemma}

\theoremstyle{definition}
\newtheorem{definition}{Definition}

\newtheorem{example}{Example}
\newtheorem{notation}{Notation}

\newcommand{\gpv}{\operatorname{GPV}}
\newcommand{\de}{\operatorname{deg}}
\newcommand{\kh}{\operatorname{KH}}
\newcommand{\rank}
{\operatorname{rank}}

\begin{document}
\begin{abstract}
In this paper, we introduce a new nontrivial filtration, called F-order, for classical and virtual knot invariants; this filtration produces filtered knot invariants, which are called finite type invariants similar to Vassiliev knot invariants.  Finite type invariants introduced by Goussarov, Polyak, and Viro are well-known, and we call them finite type invariants of GPV-order.   
We show that for any positive integer $n$ and for any classical knot $K$, there exist infinitely many of nontrivial classical knots, all of whose finite type invariants of GPV-order $\le n-1$, coincide with those of $K$ (Theorem~\ref{thmGPV}).
Further, we show that for any positive integer $n$, there exists a nontrivial virtual knot whose finite type invariants of our F-order $\le n-1$ coincide with those of the trivial knot (Theorem~\ref{thmFn}).  In order to prove Theorem~\ref{thmGPV} (Theorem~\ref{thmFn}, resp.), we define an $n$-triviality via a certain unknotting operation, called virtualization (forbidden moves, resp.), and for any positive integer $n$, find an $n$-trivial classical knot (virtual knot, resp.).  
\end{abstract}
\keywords{Finite type invariants; Knots; Virtual knots; Unknotting operations; Virtualizations; Forbidden moves \\
MSC2010: 57M27; 57M25}
\author[N.~Ito]{Noboru Ito}
\address{Graduate School of Mathematical Sciences, The University of Tokyo, 3-8-1, Komaba, Meguro-ku, Tokyo 153-8914, Japan}
\email{noboru@ms.u-tokyo.ac.jp}
\author[M.~Sakurai]{Migiwa Sakurai}
\address{National Institute of Technology, Ibaraki College, 866 Nakae, Hitachinaka-shi, Ibaraki 312-8508, Japan}
\email{migiwa@gm.ibaraki-ct.ac.jp}
\title[On $n$-trivialities for some unknotting operations]{On $n$-trivialities of classical and virtual knots for some unknotting operations}
\maketitle
\section{Introduction}
In this paper, we introduce a new nontrivial filtration for classical and virtual knot invariants.  
Here, a classical knot is an embedded circle in $\mathbb{R}^3$ and a virtual knot is a stable equivalence class of an image of a regular projection to a closed surface of an embedded circle in $\Sigma \times I$, where $\Sigma$ is a closed orientable surface and $I$ is an interval that is homeomorphic to $[0, 1]$ (for the definitions of knots and virtual knots, see Section~\ref{sec_pre} and for stable equivalences, see Carter-Kamada-Saito~\cite{CKS}).  

In 1990, Vassiliev \cite{Va} introduced a filtered space of knot invariants via a standard unknotting operation, called a \emph{crossing change}, which is an exchange of the role of an over path and an under path of a crossing of a knot diagram.   In 2000, Goussarov, Polyak, and Viro \cite{GPV} introduced another filtration, called \emph{GPV-order}, by using another unknotting operation, called \emph{virtualization} (Fig.~\ref{fig:virtualization}), for classical and virtual knots.  
From the theory by Goussarov, Polyak, and Viro \cite{GPV}, we have another framework to obtain concrete Vassiliev invariants from the dual spaces generated by chord diagrams.  We also note that from Goussarov-Polyak-Viro \cite{GPV}, Vassiliev invariants are extended to virtual knots.  In this paper, filtered knot invariants, which are similar to Vassiliev invariants and Goussarov, Polyak, and Viro invariants are called \emph{finite type invariants}.  
In 1990, the notion of \emph{$n$-triviality} was introduced by Ohyama \cite{Ohyama1} and in 1992, Taniyama \cite{Taniyama} generalized it to \emph{$n$-similarity}.  In Vassiliev theory, the notion of $n$-triviality, which is a special case of $n$-similarity, plays a significant role; here, a relationship between local moves and finite type invariants is obtained.  
For example, by Goussarov \cite{goussarov1}, for any positive integer $n$, a knot $K$ is $n$-trivial ($n$-similar to $L$, resp.) if and only if $v_m (K)$ $=$ $v_m ({\textrm{unknot}})$ ($v_m (K)$ $=$ $v_m (L)$, resp.) ($m \le n-1$).  

However, to the best of our knowledge, for an integer $n$ ($> 2$), an example of $n$-trivial classical or virtual knot of GPV-order is still not known.  
The notion of $n$-triviality is defined as follows (see Ohyama \cite{Ohyama2}).  Let $\mathcal{A}$ be a collection of $n$ pairwise disjoint, nonempty subsets, each of which consists of isolated sufficiently small disks on which unknotting operations are applied.  For any subset $T$ of the power set of $\mathcal{A}$, a knot diagram is denoted by $K(T)$ by applying a certain unknotting operation to each small disk in $T$.  Then, a knot $K$ is \emph{$n$-trivial} if there exists $T$ such that $K(\emptyset)$ is a diagram of $K$ and $K(T)$ is an unknot diagram (note that, in Goussarov-Polyak-Viro~\cite{GPV}, a slightly different definition of $n$-triviality is obtained).    

In this paper, we show that for any positive integer $n$ and for any classical knot $K$, there exist infinitely many of nontrivial classical knots, all of whose finite type invariants of GPV-order $\le n-1$, coincide with those of $K$ (Theorem~\ref{thmGPV}).

Second, we focus on an unknotting operation, called forbidden moves (Kanenobu \cite{kanenobu2} and  Nelson \cite{nelson} independently showed that this move is an unknotting operation).
By using forbidden moves, 
we introduce a new filtration, called \emph{F-order}, for classical and virtual knot invariants (Definition~\ref{f-order}).
We show that for any positive integer $n$, there exists a nontrivial virtual knot whose finite type invariants of F-order $\le n-1$ coincide with those of the trivial knot (Theorem~\ref{thmFn}).  
\section{Preliminaries}\label{sec_pre}
Suppose that $f : {\mathbb S}^1 \rightarrow {\mathbb R}^3$ is a smooth embedding.  An image
$f(\mathbb S^1)$ is called a \emph{knot} or a \emph{classical knot}.   
Two knots $K_0$ and $K_1$ are \emph{isotopic}
if there exists an isotopy $h_t : {\mathbb R}^3 \rightarrow {\mathbb R}^3, t \in [0, 1]$, with $h_0 = \operatorname{id}$ and $h_1 (K_0)$ $=$ $K_1$.   
A \emph{long knot} is a smooth embedding $\mathbb{R}$ $\to$ $\mathbb{R}^3$ which coincides with the standard embedding outside a compact set.  
An \emph{isotopy} of long knots is a smooth isotopy consisting of a family of embeddings as above.  
Let $K$ be a knot (long knot, resp.).  
Let $p$ be a regular projection ${\mathbb R}^3 \rightarrow {\mathbb R}^2$, i.e., $p(K)$ is a generic immersed plane closed curve (long curve).  In particular, every singular point of the image $p(K)$ is a transverse double point.  The image $p(K) (\subset \mathbb{R}^2)$, up to plane isotopy, with over/under information of each double point is called a \emph{diagram} of $K$.

In the classical knot theory, long knots are introduced for purely technical reasons since adding the point at the infinity point turns a long knot into a knot in the sphere $\mathbb{S}^3$ and this construction obtains an one-to-one correspondence between the isotopy classes of knots and the isotopy classes of long knots.  

A \emph{virtual knot diagram} is a knot diagram having virtual crossings (right) as well as real crossings in Fig.~\ref{fig:crossings} (left). \emph{Two virtual knot diagrams are equivalent} if one can be obtained from the other by a finite sequence of \emph{generalized Reidemeister moves} in Fig.~\ref{fig:GRM}. The equivalence class of virtual knot diagrams modulo the generalized Reidemeister moves is called a {\em virtual knot}.  Similarly, a \emph{long virtual knot diagram} is a long knot diagram having virtual crossings as well as real crossings. A {\em long virtual knot} is an equivalence class of long virtual knot diagrams modulo the generalized Reidemeister moves.

\begin{figure}[htbp]
\centering
\includegraphics[width=3cm]
%,clip]
{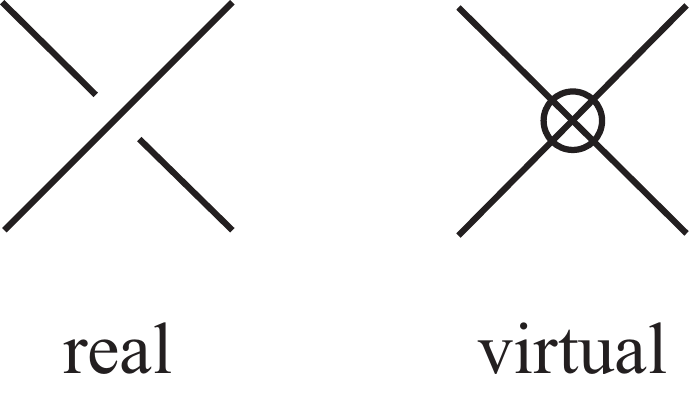}
\caption{Crossing types.}
\label{fig:crossings}
\end{figure}

\begin{figure}[htbp]
\centering
\includegraphics[width=9cm,clip]{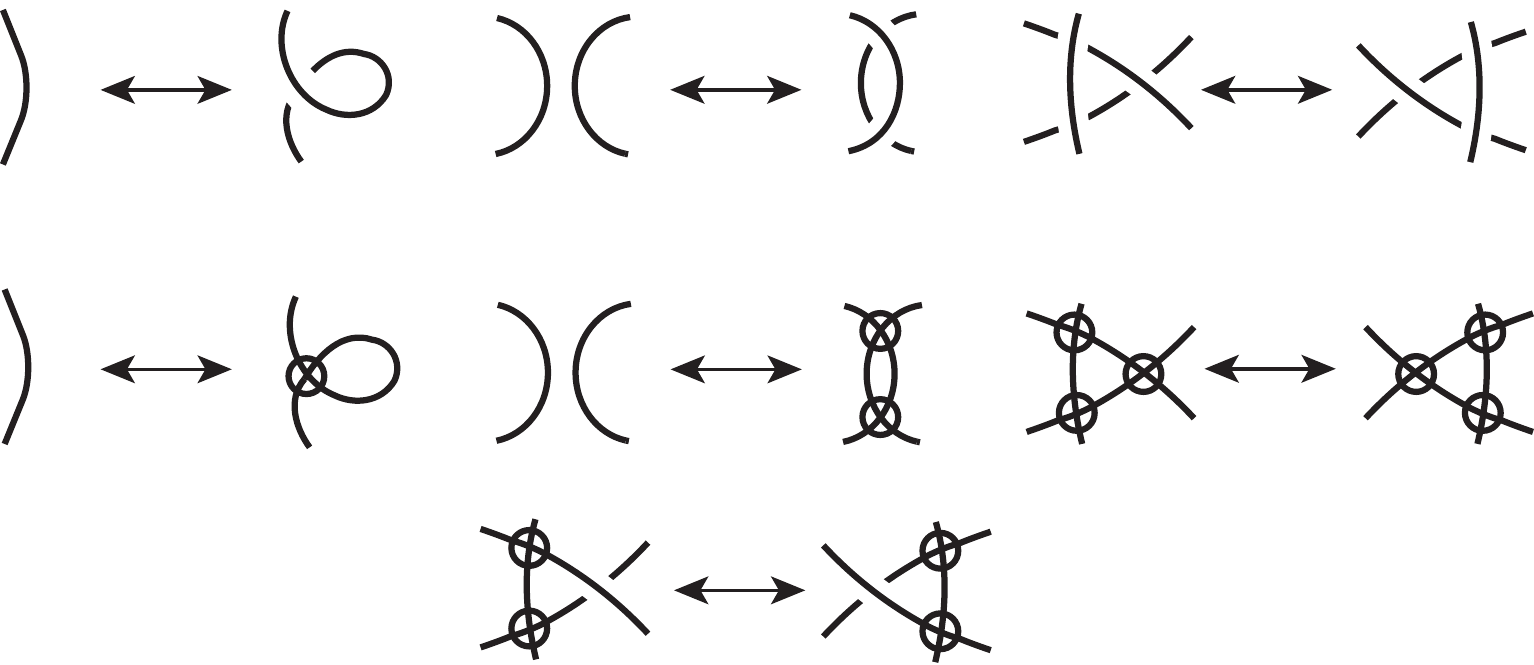}
\caption{Generalized Reidemeister moves.}
\label{fig:GRM}
\end{figure}

A virtual knot diagram (long virtual knot diagram, resp.) is regarded as the image of an immersion from $\mathbb S^1$ ($\mathbb R$, resp.) into $\mathbb R^2$. Let $K$ be a virtual knot diagram (long virtual knot diagram, resp.).  A {\em Gauss diagram} for $K$ is the preimage of $K$ with chords, each of which connects the preimages of each real crossing.  We specify over/under information of each real crossing on the corresponding chord by directing the chord toward the under path and decorating each chord with the sign of the crossing (Fig.~\ref{fig:sign}).

It is well-known that there exists a bijection from the set of virtual knots (long virtual knots, resp.) to the set of equivalence classes of their Gauss diagrams modulo the \emph{generalized Reidemeister moves of Gauss diagrams} in Fig.~\ref{fig:D_GRM} (Fig.~\ref{fig:D_GRM_long}, resp.).  In this paper, we identify a virtual knot (long virtual knot, resp.) with an equivalence class of Gauss diagrams, and we freely use either one of them depending on situations.

\begin{figure}[htbp]
\centering
\includegraphics[width=3cm,clip]{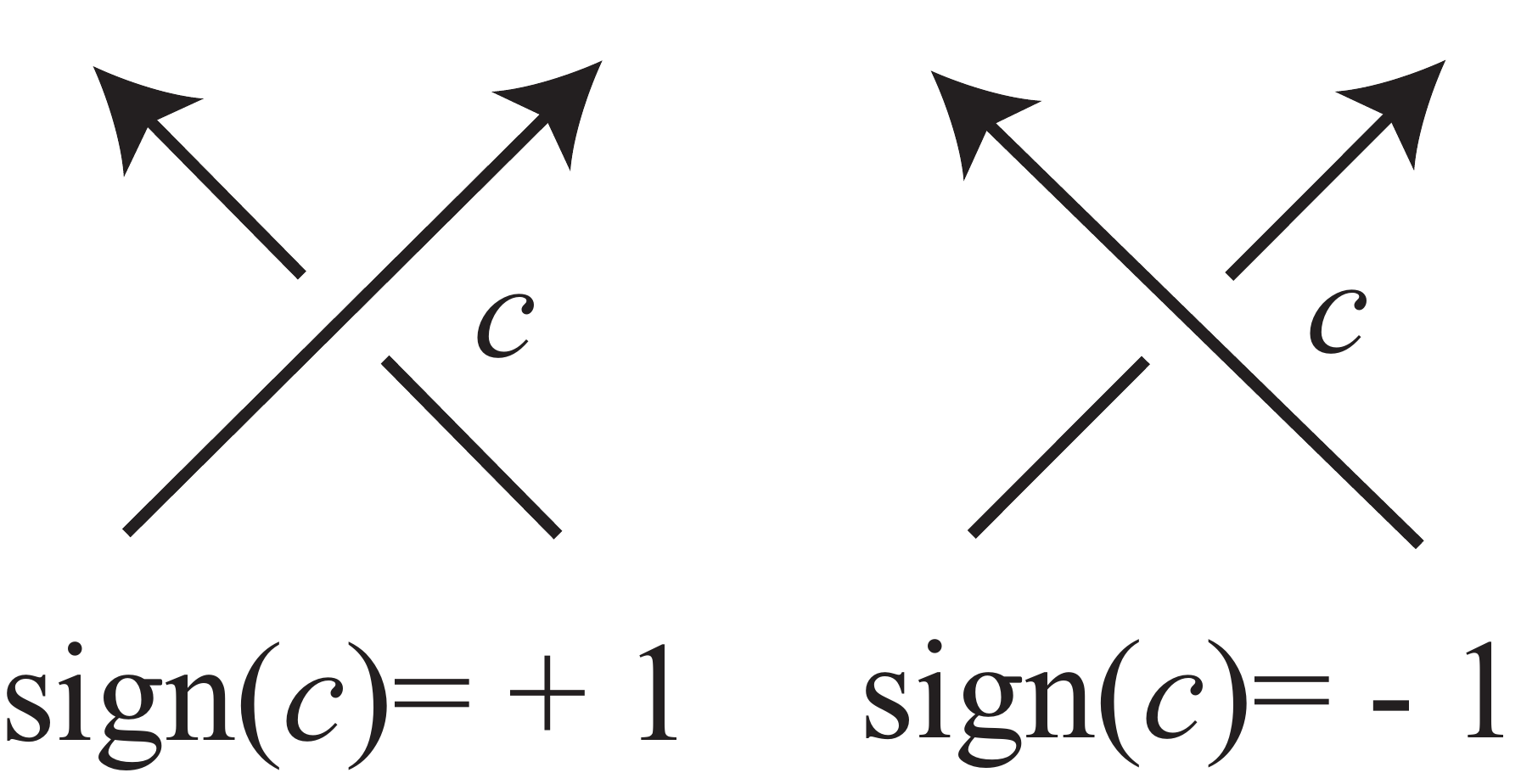}
\caption{The sign of a real crossing.}
\label{fig:sign}
\end{figure}

\begin{figure}[htbp]
\centering
\includegraphics[width=6cm,clip]{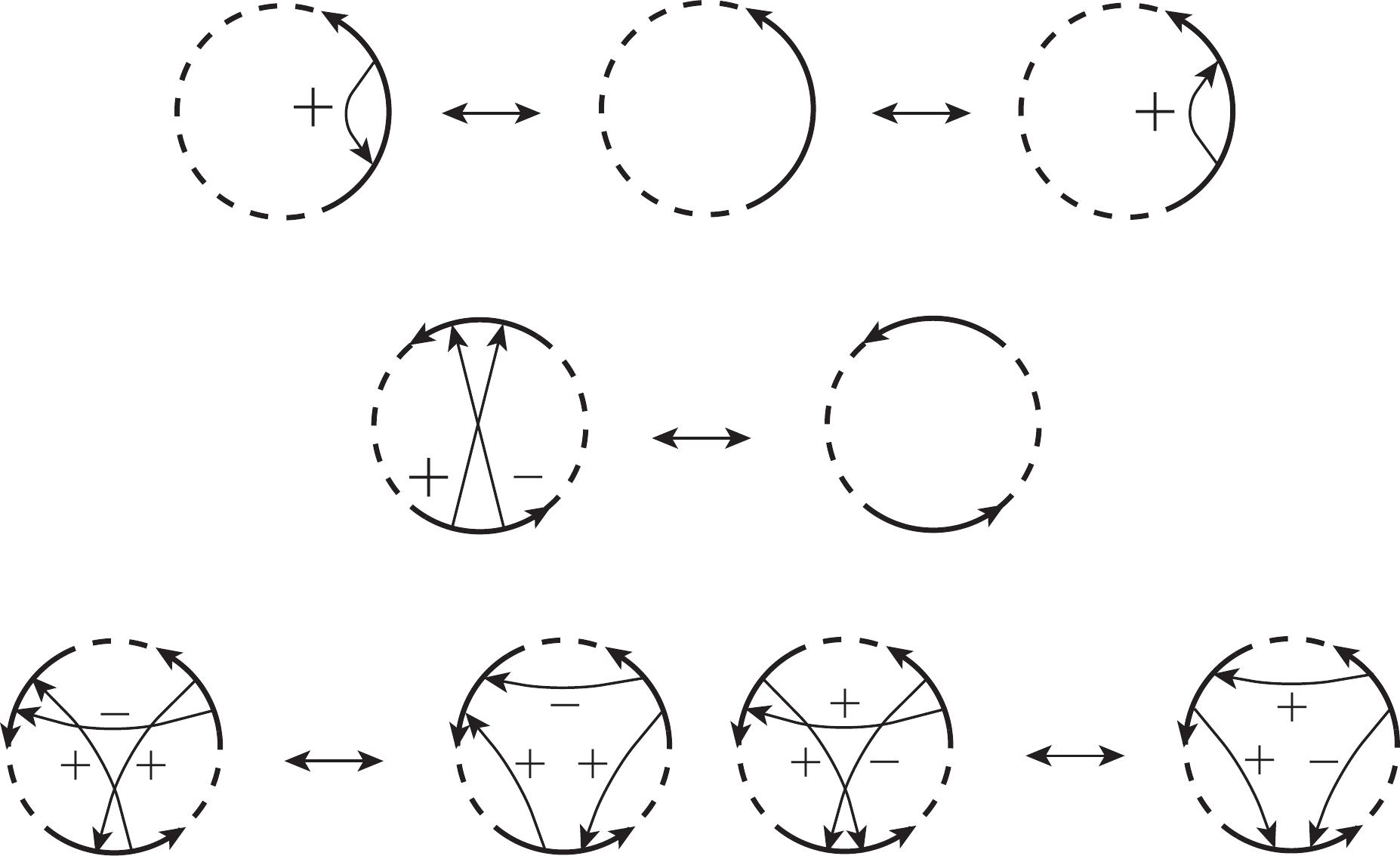}
\caption{Generalized Reidemeister moves of Gauss diagrams for virtual knots.}
\label{fig:D_GRM}
\end{figure}
\begin{figure}[htbp]
\centering
\includegraphics[width=10cm, clip]{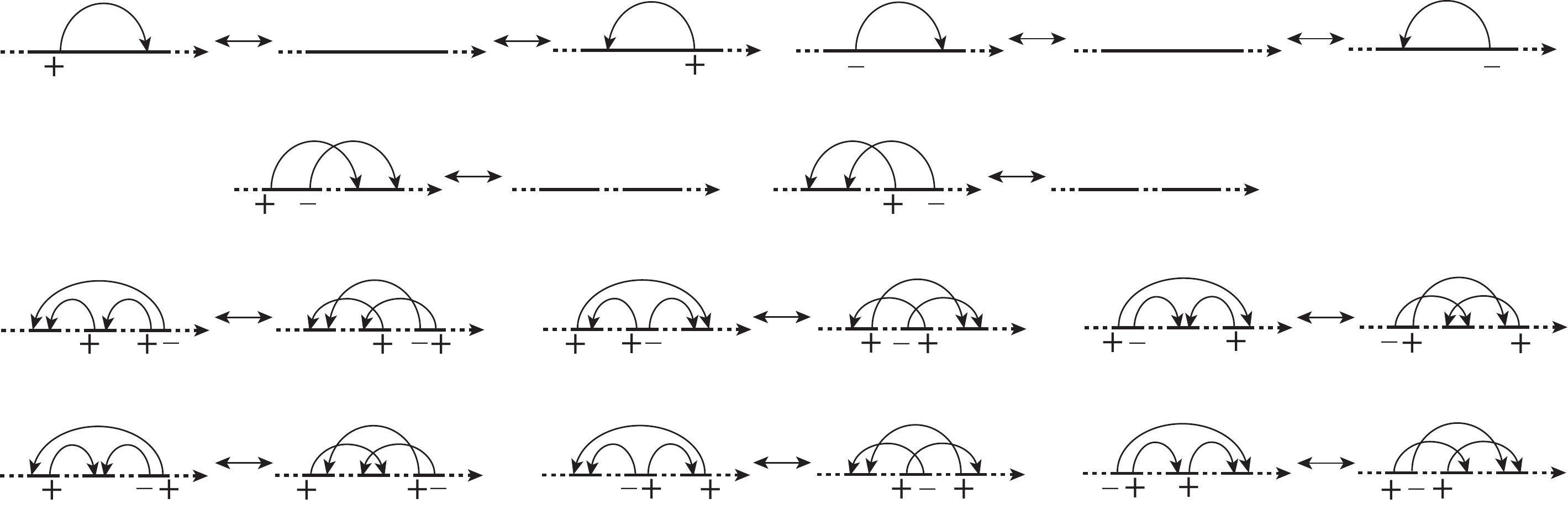}
\caption{Generalized Reidemeister moves of Gauss diagrams for long virtual knots.}\label{fig:D_GRM_long}
\end{figure}

An \emph{arrow diagram} is just a Gauss diagram with all chords drawn dashed. Let $\mathcal A$ be the set of all arrow diagrams, and $\mathcal D$ the set of all Gauss diagrams. A \emph{subdiagram} of $D \in \mathcal D$ is a Gauss diagram consisting of a subset of the chords of $D$. Define a map $i : \mathcal D \rightarrow \mathbb{Z} \mathcal A$ by the map that makes all the chords of a Gauss diagram dashed, and define a map $I:{\mathcal D} \rightarrow {\mathbb{Z} \mathcal A}$ by
\[
I\left( D\right) = \sum_{D' \subset D}i(D'),
\]
where the sum is over all subdiagrams of $D$. Extend the map $I$ to $\mathbb{Z} \mathcal D$ linearly. On the generators of $\mathbb{Z} \mathcal A$, define $( D, E )$ to be $1$ if $D = E$ and $0$ otherwise, and then extend $( \cdot , \cdot )$ bilinearly. Put
\begin{equation}\label{eq1}
	\langle  A,D\rangle =\left( A,I\left( D\right) \right),
\end{equation}
for any $D \in \mathcal D$ and $A \in \mathbb{Z} \mathcal A$ (for an example, see Examples~\ref{pr:GPV} and \ref{example2}).  

A \emph{trivial knot diagram} is a virtual knot diagram (long virtual knot diagram, resp.) with no double points.   Let $K$ be a virtual knot whose diagram is $D$.  Suppose that a virtual knot diagram $D$ and a trivial knot diagram are equivalent.  Then $K$ is called a \emph{trivial knot} or an \emph{unknot}.  Let $D$ be a long knot diagram or a knot diagram.  Then, a \emph{local move} is a replacement of a sufficiently small disk $d (\subset {\mathbb{R}}^2)$ on $D (\subset {\mathbb{R}^2})$ by another disk $d' (\subset {\mathbb{R}^2})$ such that $\partial d$ $=$ $\partial d'$ and $(\mathbb{R}^2 \setminus d) \cup d'$ gives a diagram.   
For a virtual knot diagram (long virtual knot diagram, resp.), the type of local move shown in Fig.~\ref{fig:virtualization} is called {\em virtualization} and each type of local moves shown in Fig.~\ref{fig:forbidden} is called {\em forbidden moves}. 
Let $M$ be a type of local moves. If any virtual knot diagram (long virtual knot diagram, resp.) can be transformed into a trivial knot diagram by a finite sequence of local moves of type $M$ and generalized Reidemeister moves, moves of type $M$ are called {\it unknotting operations} for virtual knots (long virtual knots, resp.). 

\begin{figure}[htbp]
\centering
\includegraphics[width=4cm,clip]{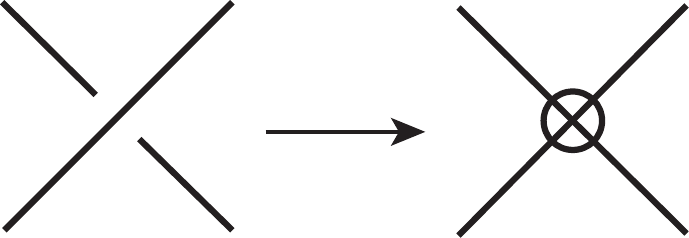}
\caption{Virtualization.}
\label{fig:virtualization}
\end{figure}

\begin{figure}[htbp]
\centering
\includegraphics[width=9.5cm]{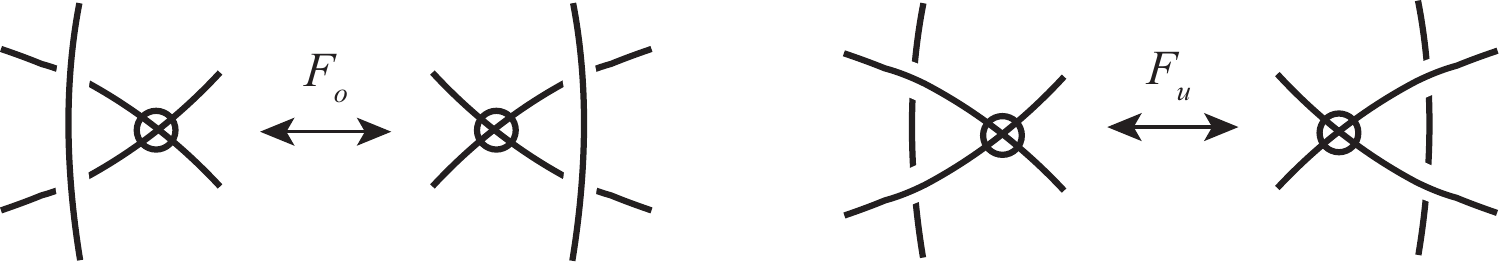}
\caption{Forbidden moves consisting of type $F_o$ and type $F_u$.}
\label{fig:forbidden}
\end{figure}

\begin{fact}[Goussarov, Polyak, and Viro \cite{GPV}]
Moves of virtualization are unknotting operations for virtual knots and for long virtual knots.
\end{fact}
\begin{fact}[Kanenobu \cite{kanenobu1}, Nelson \cite{nelson}]\label{fact2}
Forbidden moves are unknotting operations for virtual knots and for long virtual knots.  
\end{fact}

\begin{definition}[triangle]\label{dfn_triangle}
Suppose that $D$ is a virtual knot diagram or a long virtual knot diagram ($\subset \mathbb{R}^2$).  For $D$, if there exists a disk ($\subset \mathbb{R}^2$) that look like one of Fig.~\ref{fig:triangle disks}, the disk is called a {\it triangle}. Signs of triangles are defined as shown in Fig.~\ref{fig:signs of triangle disks}.  When we would like to specify the sign of a triangle, the triangle is denoted by an $\epsilon$-triangle ($\epsilon = +, -$).   
A triangle with the opposite sign to $\epsilon$ is called a $(-\epsilon)$-triangle. 

\begin{figure}[htbp]
\centering
\includegraphics[width=11cm]{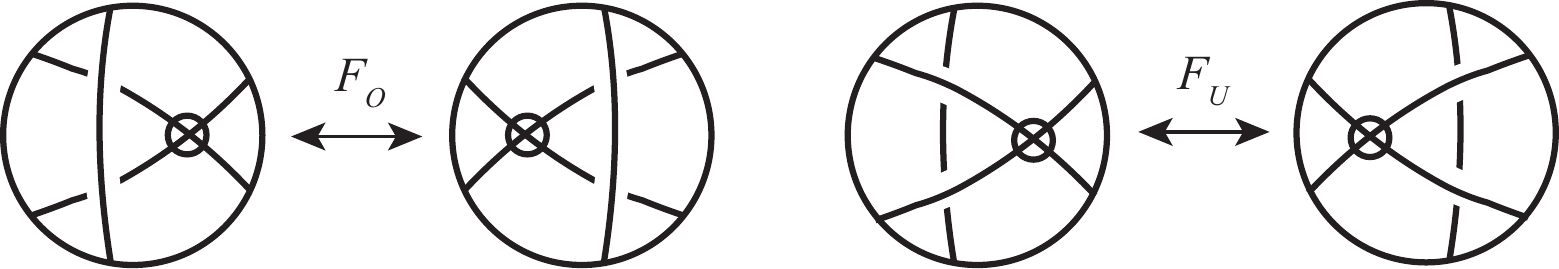}
\caption{Triangles (four types).}
\label{fig:triangle disks}
\end{figure}

\begin{figure}[htbp]
\centering
\includegraphics[width=12cm,clip]{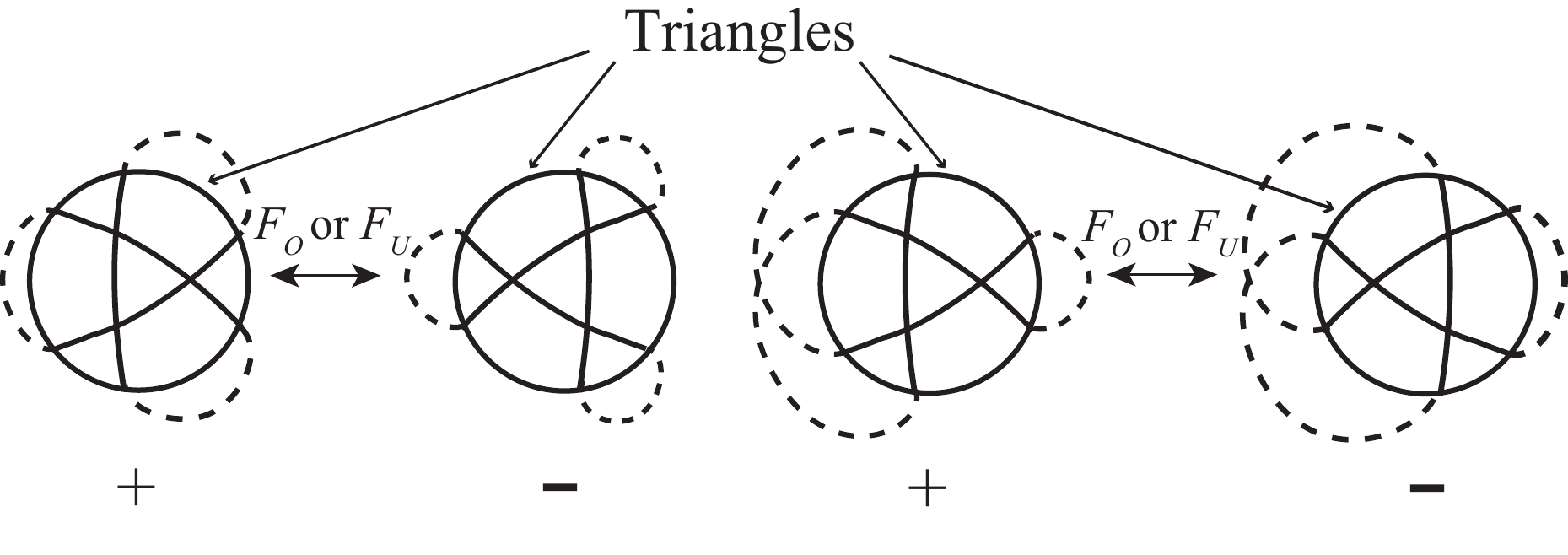}
\caption{Signs of triangles.  Dotted curves indicate the connections of virtual knot diagrams or long virtual knot diagrams.  For long virtual knot diagrams, the infinity point is on a dotted curve.}
\label{fig:signs of triangle disks}
\end{figure}

\end{definition}

\begin{definition}[Finite type invariants of GPV-order]
Let $\mathcal{VK}$ be the set of the virtual knots (long virtual knots, resp.) 
and $G$ an abelian group.  Let $v$ be a function from $\mathcal{VK}$ to $G$.  Suppose that $v$ is an invariant of virtual knots (long virtual knots, resp.).    
We say that $v: {\mathcal{VK}} \rightarrow G$ is a \emph{finite-type invariant of GPV-order} $\leq n$ if for any virtual knot diagram (long virtual knot, resp.) $D$ and for any $n+1$ real crossings $d_1, d_2, \ldots, d_{n+1}$,
\[
\sum_{\delta} {(-1)^{|\delta |} v(D_{\delta })=0},
\]
where $\delta=(\delta_1, \delta_2, \ldots , \delta_{n+1})$ runs over $(n+1)$-tuples of 0 or 1, $|\delta |$ is the number of $1$'s in $\delta$, and
$D_\delta$ is a diagram obtained from $D$ by switching every $d_i$ with $\delta_i=1$ to the virtual crossing.    

If $v$ is a finite type invariant of GPV-order $\leq n$ and is not a finite type invariant of GPV-order $\leq {n-1}$, we say that $v$ \emph{is a finite type invariant of GPV-order} $n$, and $v$ is denoted by $v^{\gpv}_n$.  
\end{definition}  
\begin{definition}[Finite type invariants of F-order]\label{f-order}
Let $\mathcal{VK}$ be the set of the virtual knots (long virtual knots, resp.) 
and $G$ an abelian group.  
Let $v$ be a function from $\mathcal{VK}$ to $G$.  Suppose that $v$ is an invariant of virtual knots (long virtual knots, resp.).  
We say that $v :$ ${\mathcal{VK}} \rightarrow G$ is a \emph{finite-type invariant for forbidden moves of order} $\leq n$ if for any virtual knot diagram (long virtual knot diagram, resp.) $D$ and for any $n+1$ disjoint triangles, each of which is $\epsilon_i$-triangle $d_i$ ($1 \le i \le n+1$, $\epsilon_i = -1$ or $1$),
\[
\sum_{\delta} {(-1)^{|\delta |} v(D_{\delta })=0},
\]
where $\delta=(\delta_1, \delta_2, \ldots , \delta_{n+1})$ runs over $(n+1)$-tuples of 0 or 1, $|\delta |$ is the number of $1$'s in $\delta$, and
$D_\delta$ is a diagram obtained from $D$ by switching every $\epsilon_i$-triangle $d_i$ with $\delta_i=1$ to the $(-\epsilon_i)$-triangle which is obtained from $d_i$ by a single forbidden move.

In this paper, a finite-type invariant for forbidden moves of order $\leq n$ is simply called a \emph{finite-type invariant of F-order} $\leq n$.  
If $v$ is a finite type invariant of F-order $\leq n$ and is not a finite type invariant of F-order $\leq {n-1}$, we say that $v$ \emph{is a finite type invariant of F-order} $n$, and $v$ is denoted by $v^F_n$.    
\end{definition}
\begin{example}[cf.~Goussarov-Polyak-Viro~\cite{GPV}, Polyak-Viro~\cite{PV}]
\label{pr:GPV} 
Let $v_{2, 1}$ and $v_{2, 2}$ be finite type invariants of GPV-order $2$ of long virtual knots.  
Recall that the definition $\langle \cdot, \cdot \rangle$ by (\ref{eq1}).  
By Goussarov, Polyak, and Viro \cite{GPV}, these are obtained by
\[
v_{2,1}(\cdot)=\biggl\langle \sum_{\varepsilon_1,\varepsilon_2}\varepsilon_1\varepsilon_2 \parbox{50pt}{\includegraphics[width=50pt]{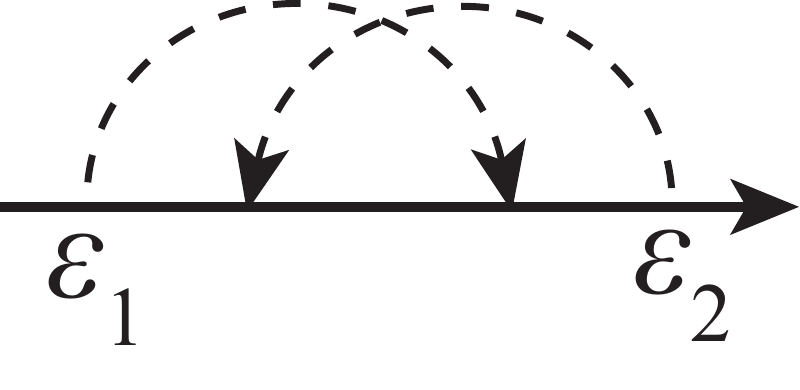}},\cdot \biggr\rangle \mbox{ and }
v_{2,2}(\cdot)=\biggl\langle \sum_{\varepsilon_1,\varepsilon_2}\varepsilon_1\varepsilon_2 \parbox{50pt}{\includegraphics[width=50pt]{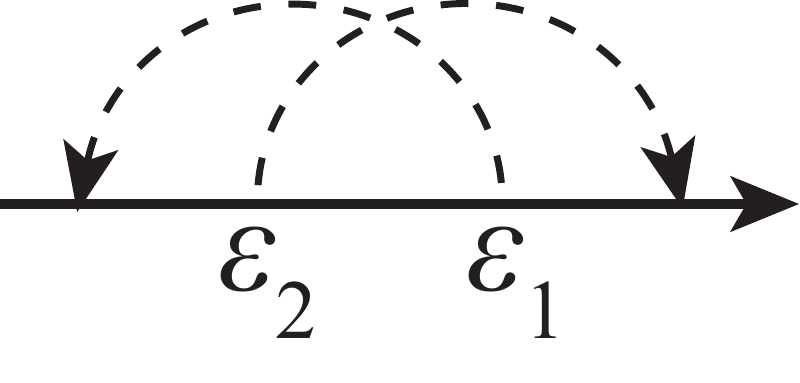}},\cdot \biggr\rangle .
\]    
It is known that for any classical knot $K$, $v_{2, 1} (K)$ $=$ $v_{2, 2} (K)$ and both $v_{2, 1}$ and $v_{2, 2}$ correspond to the second coefficient of Conway polynomial where $v_{2, 1} ({\textrm{unknot}})$ $=$ $0$ and $v_{2, 1} ({\textrm{trefoil}})$ $=$ $1$ (see Polyak-Viro~\cite{PV}).  It is clear that both $v_{2, 1}$ and $v_{2, 2}$ are nontrivial knot invariants, which implies that $v_{2, 1}$ ($v_{2, 2}$, resp.) is not a finite type invariant of F-order $0$.   
Note that every finite type invariant of F-order $0$ is a constant map (traditionally, such an invariant is called a \emph{trivial invariant}) by Fact~\ref{fact2}.   
In the following, we show that $v_{2, 1}$ ($v_{2, 2}$, resp.) is a finite type invariant of F-order $1$.  It is sufficient to show that $v_{2, 1}$ ($v_{2, 2}$, resp.) is a finite type invariant of F-order $\le 1$.  

A single forbidden move is presented by Gauss diagrams up to signs of crossings as in Fig.~\ref{FigA}.  
\begin{figure}[htbp]
\includegraphics[width=8cm]{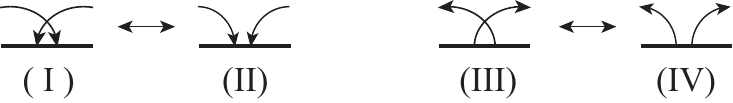}
\caption{Gauss diagrams and a single forbidden move.}\label{FigA}
\end{figure}
Recall that the definition of $D_\delta$ in Definition~\ref{f-order}.    
Suppose that a long virtual knot diagram $D$ has two disjoint triangles.  Let $D_{(0, 0)}$ $=$ $D$ and let $\delta$ $=$ $(\delta_1, \delta_2)$.  Let $J$ $=$ $\sum_{\delta} (-1)^{|\delta|} v_{2, 1} (D_{\delta})$.  
By definition, $D_{\delta}$ with $|\delta|$ $=$ $1$ ($D_{(1, 1)}$, resp.) is obtained from $D_{(0, 0)}$ by a single forbidden move (exactly two forbidden moves, resp.).   
If there exists $D_{\delta}$ such that type~(I) in Fig.~\ref{FigA} does not appears twice in $D_{\delta}$, it is easy to see $J$ $=$ $0$.  Thus, without loss of generality, we can suppose that $D_{(\delta_1, \delta_2)}$ is as in Fig.~\ref{FigB}.  
\begin{figure}[htbp]
\includegraphics[width=8cm]{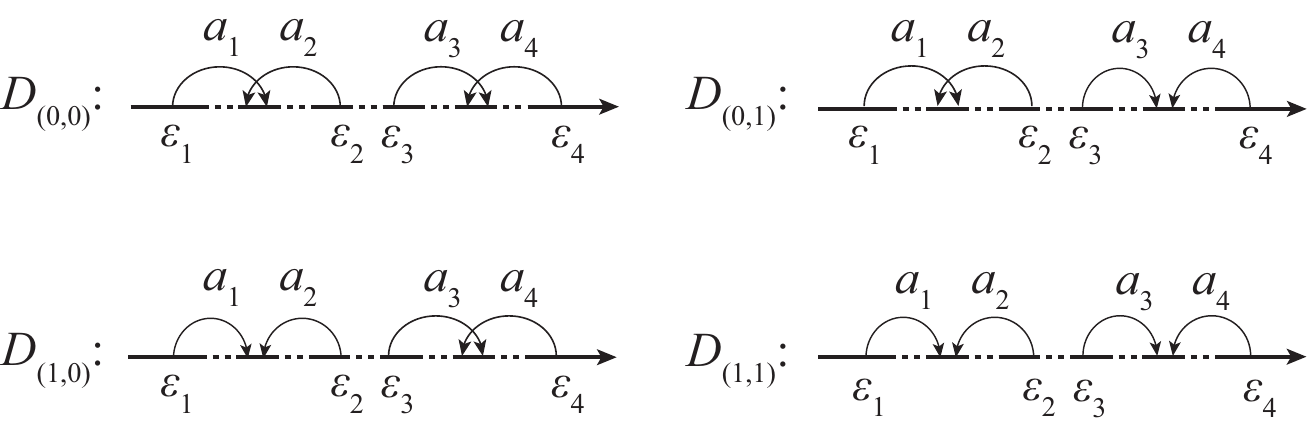}
\caption{$D_{(0, 0)}$, $D_{(0, 1)}$, $D_{(1, 0)}$, and $D_{(1, 1)}$.  $a_i$ ($1 \le i \le 4$) denotes an oriented chord and $\epsilon_i$ ($1 \le i \le 4$) denotes the sign of $a_i$.}\label{FigB}
\end{figure}
Here, for oriented chords $a_i$ ($1 \le i \le 4$) in Fig.~\ref{FigB}, it is clear that whether there is an intersection of $\alpha \in \{ a_1, a_2 \}$ with $\beta \in \{ a_3, a_4 \}$ or not does not affect the values of $J$ under the corresponding forbidden moves.   Then,  
\begin{align*}
\sum_{\delta} (-1)^{|\delta|} v_{2, 1} (D_{\delta}) &=
 v_{2, 1} (D_{(0, 0)}) - v_{2, 1} (D_{(0, 1)}) - v_{2, 1} (D_{(1, 0)}) + v_{2, 1} (D_{(1, 1)})\\
&= \left( v_{2, 1} (D_{(0, 0)}) - v_{2, 1} (D_{(0, 1)}) \right) - \left( v_{2, 1} (D_{(1, 0)}) - v_{2, 1} (D_{(1, 1)}) \right) \\
&= \epsilon_3 \epsilon_4 - \epsilon_3 \epsilon_4 \\
&= 0. 
\end{align*}
For $v_{2, 2}$, since the arguments are essentially the same as that of $v_{2, 1}$, the detail of the proof of $\sum_{\delta} (-1)^{|\delta|} v_{2, 2} (D_{\delta})$ $=$ $0$ is omitted (focus on (III) of Fig.~\ref{FigA}).  
\end{example}
\begin{example}\label{example2}
Let $v_{3, 1}$ be a finite type invariant of GPV-order $3$ of virtual knots.  
By Goussarov, Polyak, and Viro \cite{GPV}, this is obtained by
\begin{align*}
v_{3, 1}(\cdot)=&
\biggl\langle \sum_{\varepsilon_1,\varepsilon_2,\varepsilon_3}\varepsilon_1\varepsilon_2\varepsilon_3\Bigl( 3\parbox{26pt}{\includegraphics[width=26pt]{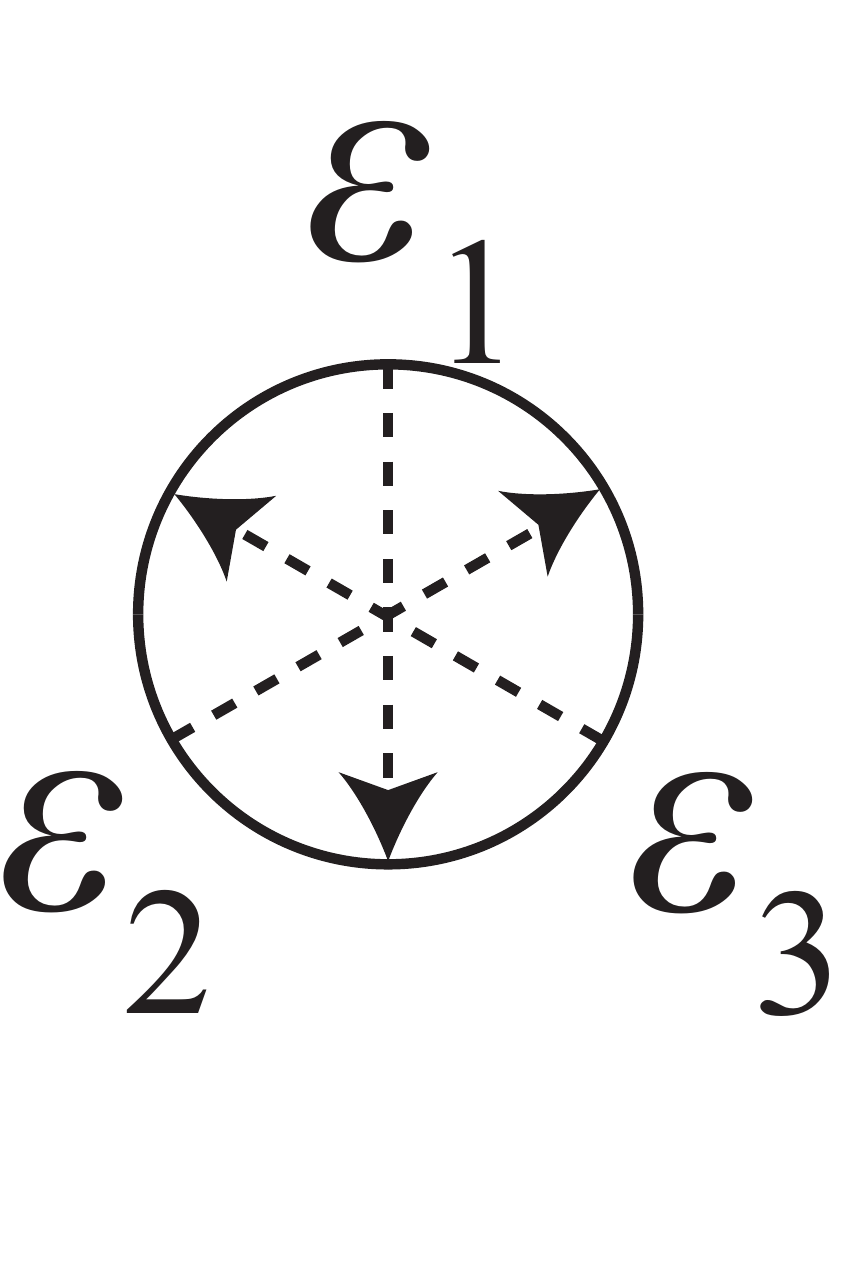}}-\parbox{26pt}{\includegraphics[width=26pt]{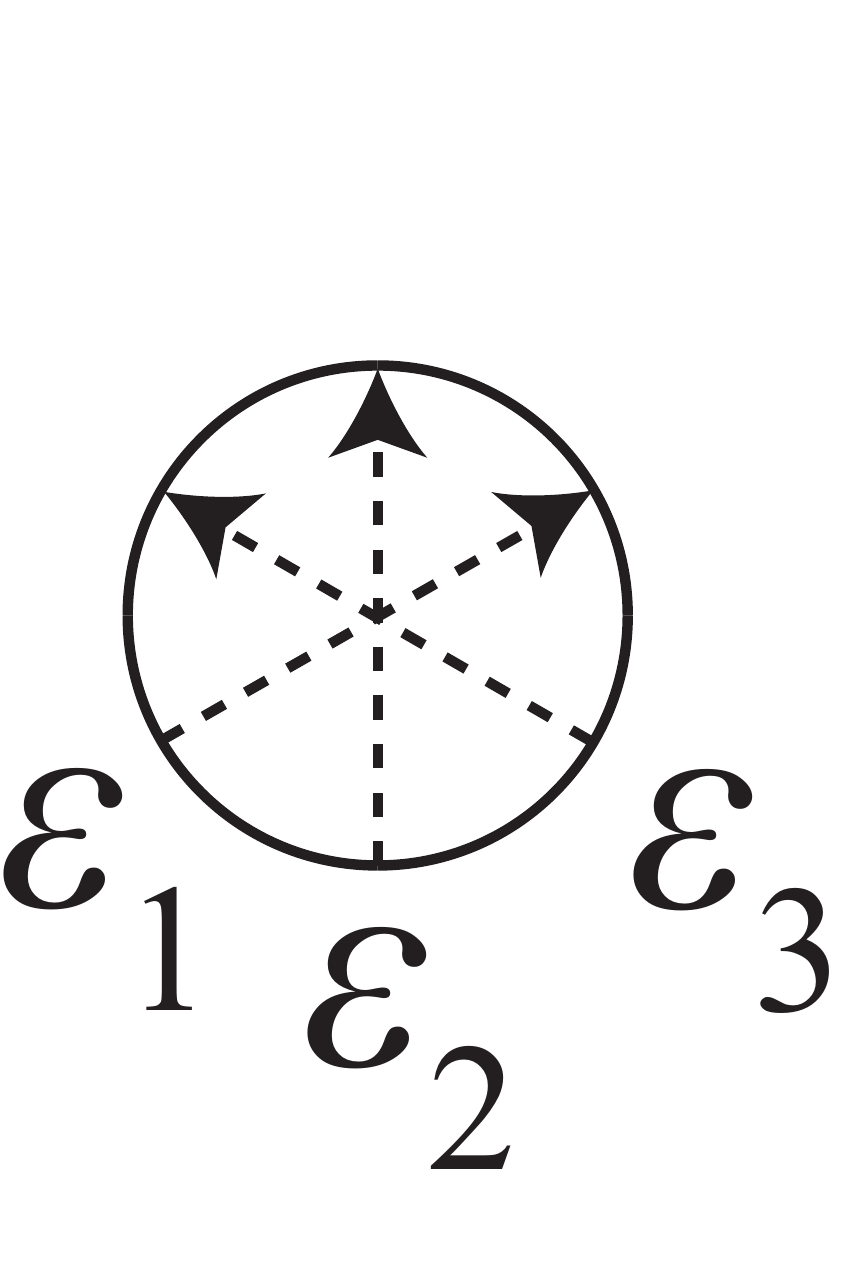}}+\parbox{26pt}{\includegraphics[width=26pt]{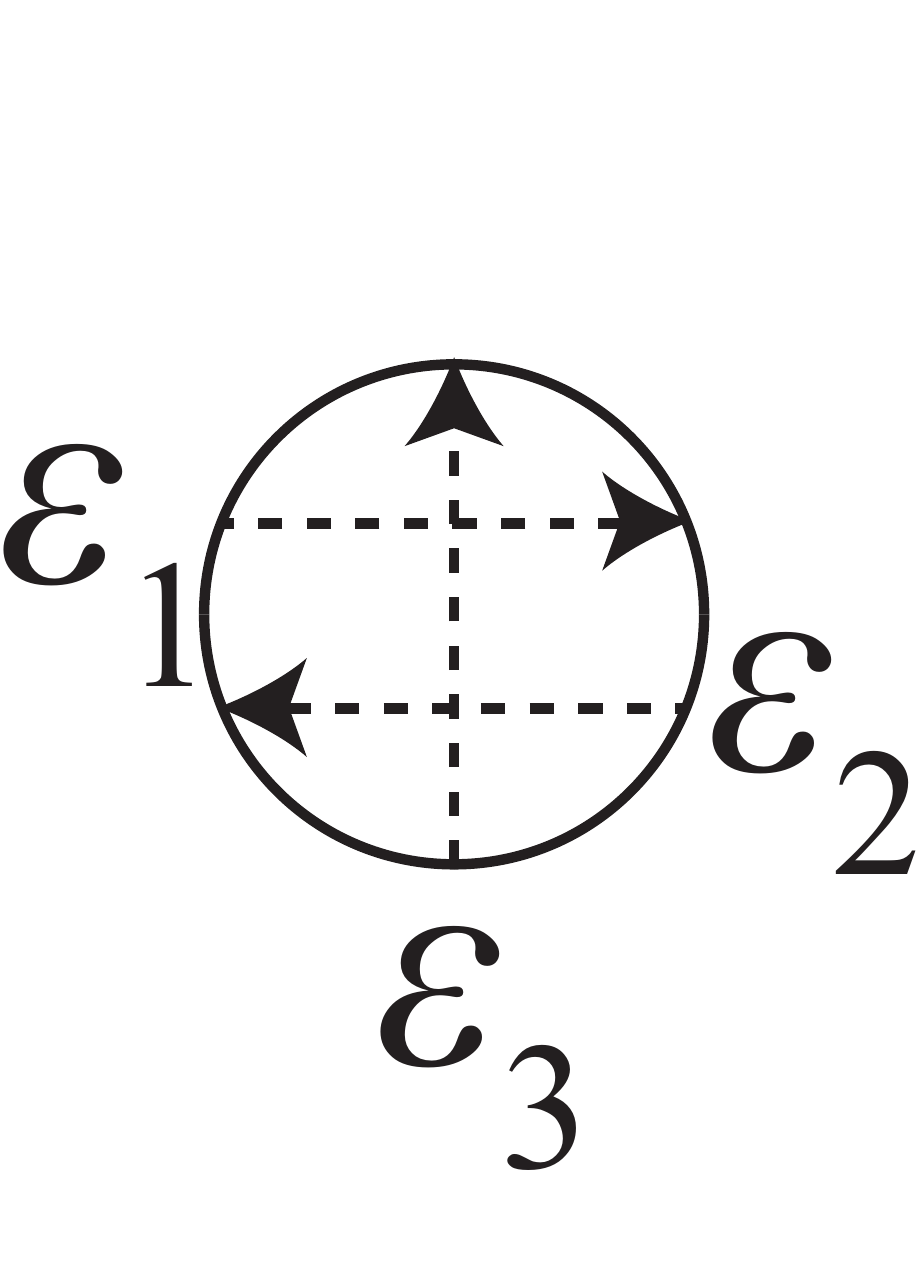}}+\parbox{26pt}{\includegraphics[width=26pt]{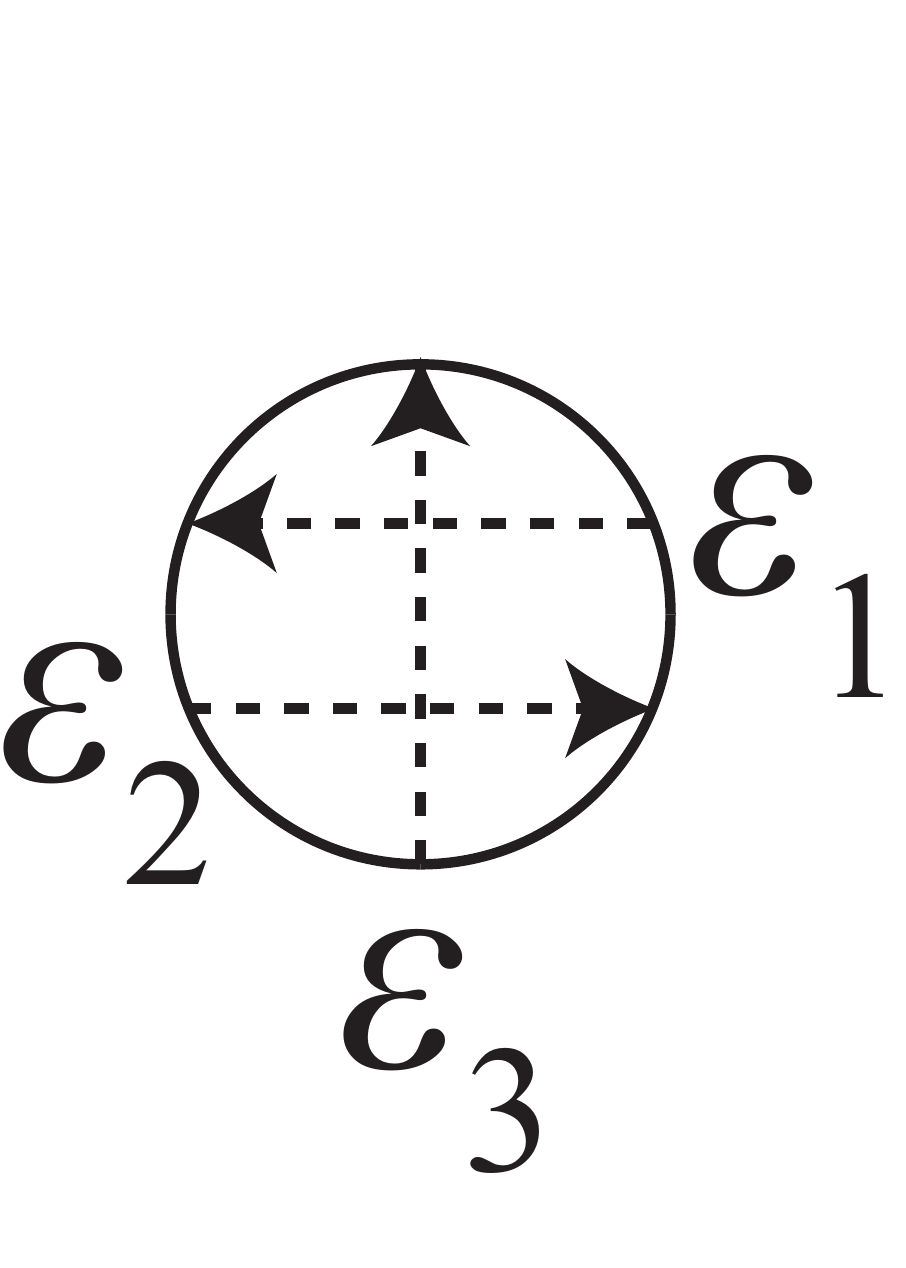}}-\parbox{26pt}{\includegraphics[width=26pt]{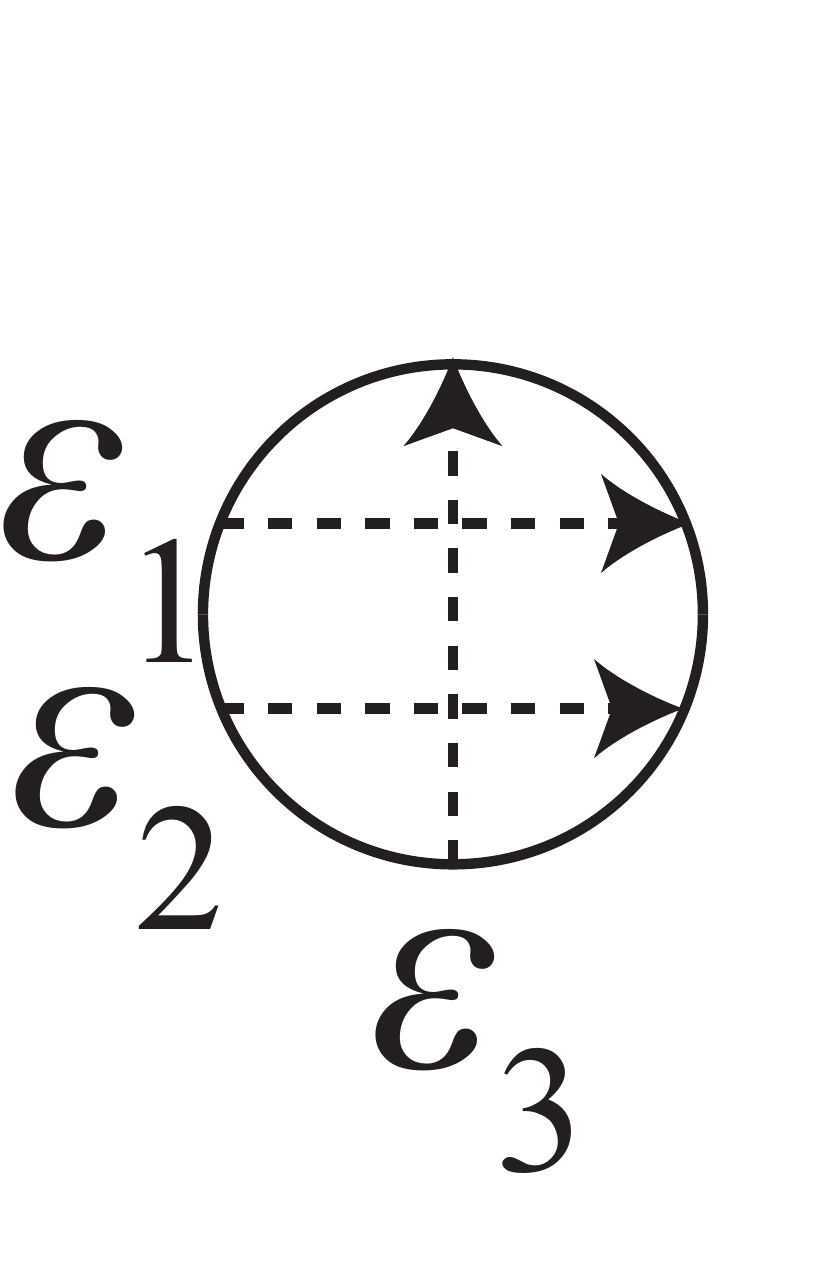}}-\parbox{26pt}{\includegraphics[width=26pt]{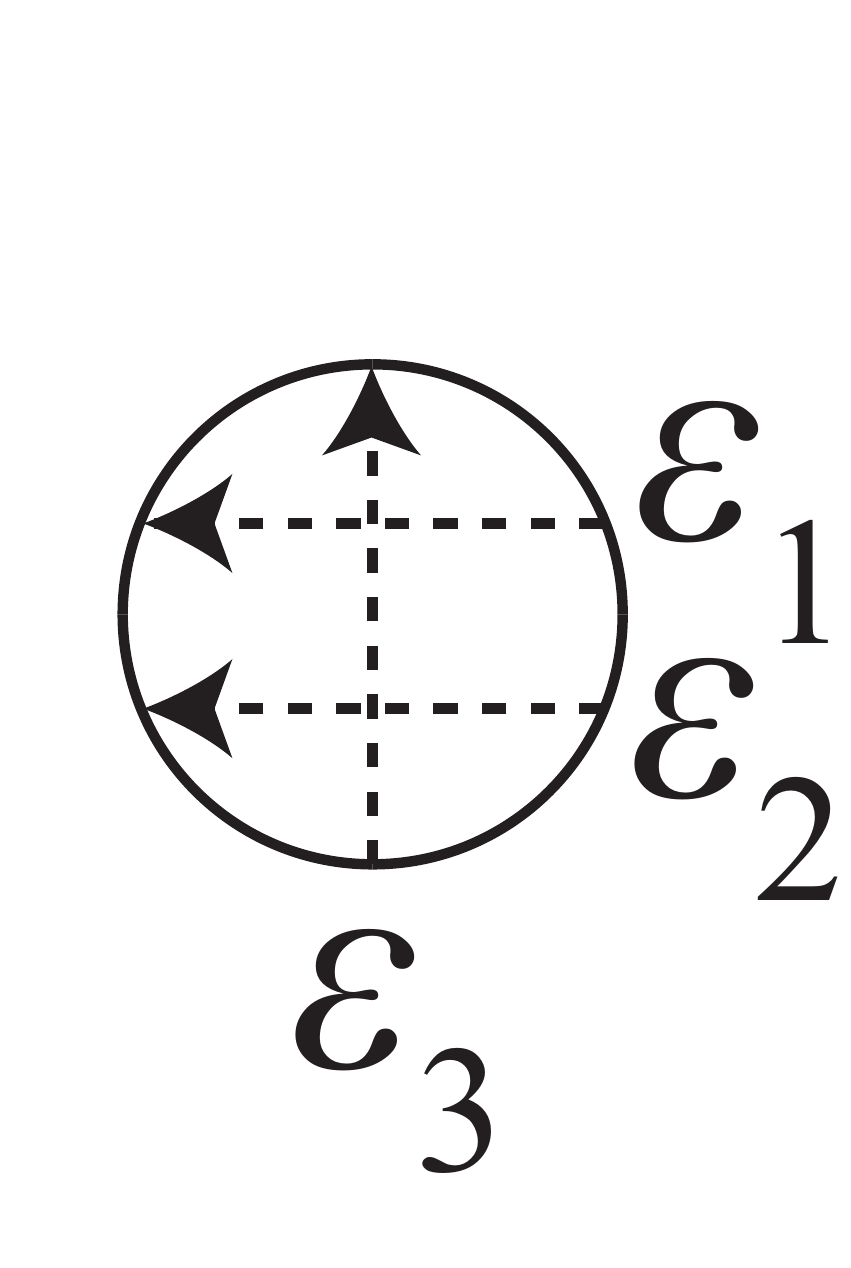}}\Bigr)\\
&-\parbox{23pt}{\includegraphics[width=23pt]{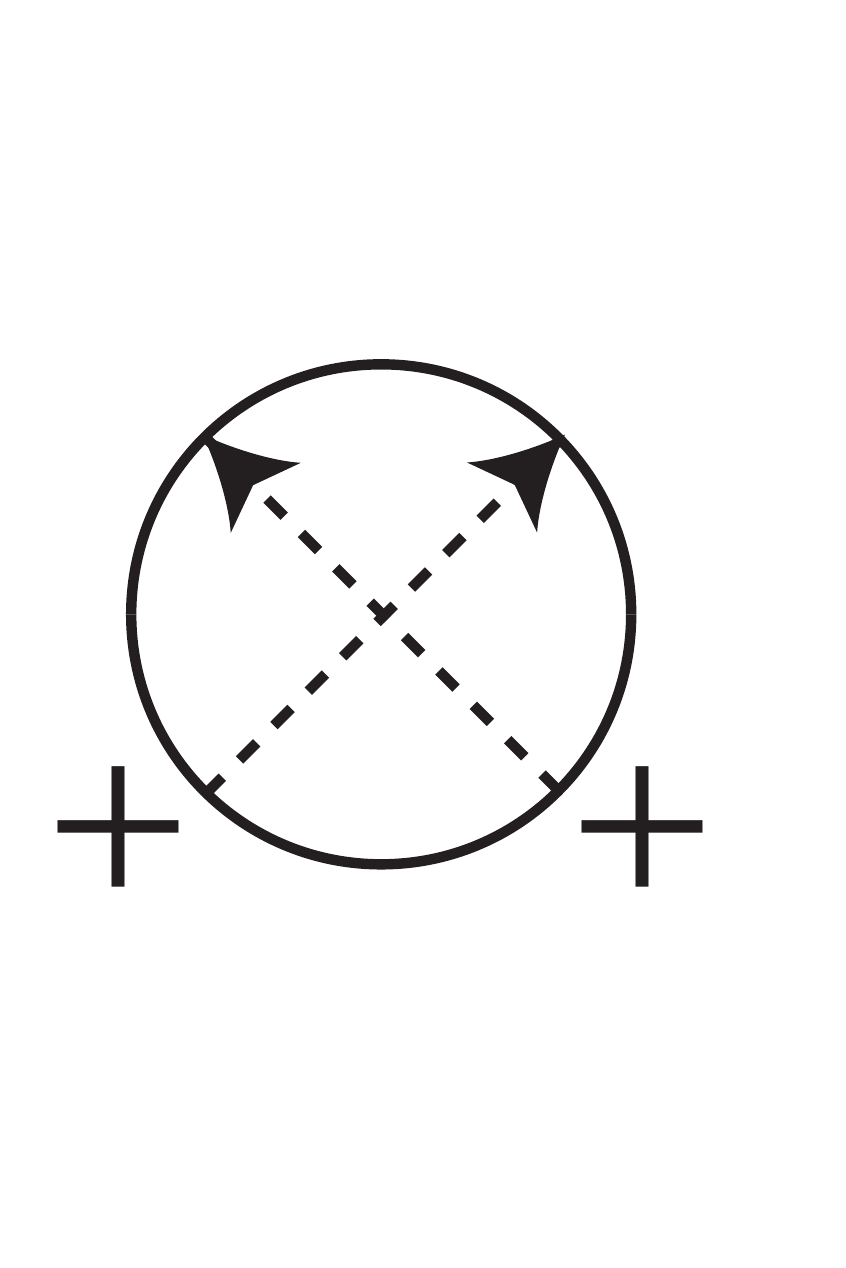}}+\parbox{23pt}{\includegraphics[width=23pt]{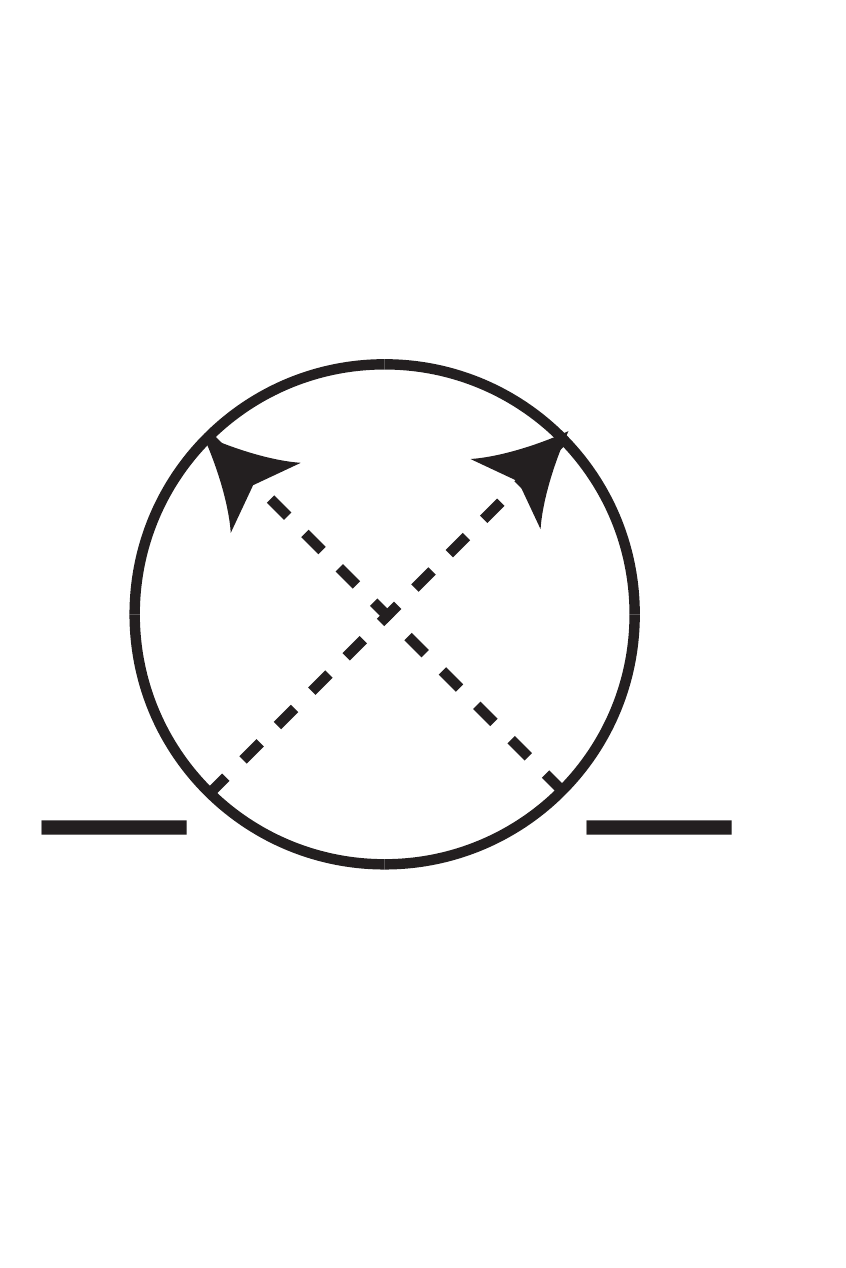}},\cdot \biggr\rangle,
\end{align*}
where $\varepsilon_i = \pm 1$ $(i=1, 2, 3)$.

Let $D_+$ be a virtual knot diagram containing a $+$-triangle.  Let $D_-$ be a virtual knot diagram containing $-$\,-triangle is obtained by a single forbidden move.  Formulas of the difference $v_{3, 1} (D_+)$ $-$ $v_{3, 1} (D_-)$ are obtained by \cite[Lemma~3.5]{Sakurai}, and using the formulae, it is elementary to show that $v_{3, 1}$ is a finite type invariant of F-order $\le 1$ in the same way as Example~\ref{pr:GPV}.  Details of this proof are left to the reader.  
Note that $v_{3, 1}$ is a nontrivial virtual knot invariant, which implies that $v_{3, 1}$ is not a finite type invariant of F-order $0$.     
Therefore, $v_3 (K)$ is a finite type invariant of F-order $1$.  
\end{example}
\begin{definition}[$GPV_n$-similar, $GPV_n$-trivial]
Let $K$ ($L$, resp.) be a virtual knot or a long virtual knot.  Let
$\widetilde{K}$ ($\widetilde{L}$, resp.) be a diagram of $K$ ($L$, resp.).  
Suppose that $A_1$, $A_2$, $\ldots$, $A_n$ are non-empty sets of crossings $\widetilde{K}$ or $A_1$, $A_2$, $\ldots$, $A_n$ are non-empty sets of crossings $\widetilde{L}$.  
Then, $K$ and $L$ are \emph{$GPV_n$-similar} 
if there exist $A_i$ ($1 \leq i \leq n$) such that 
\begin{itemize}
\item $A_i \cap A_j = \emptyset$ $(i \neq j)$, 
\item By replacing real to virtual at every crossing of any nonempty subfamily of $\{A_i~|~1\leq i \leq n\}$, $\widetilde{L}$ is obtained from $\widetilde{K}$ or $\widetilde{K}$ is obtained from $\widetilde{L}$.  
\end{itemize}
In particular, if a virtual knot $K$ and a trivial knot is $GPV_n$-similar, $K$ is \emph{$GPV_n$-trivial}.   
\end{definition}
\begin{definition}[$F_n$-similar, $F_n$-trivial]
Let $K$, $L$ be virtual knots and $\widetilde{K}$ ($\widetilde{L}$, resp.) a diagram of $K$ ($L$, resp.).  
Let $A_1$, $A_2$, $\ldots$, $A_n$ be non-empty sets of disjoint triangles in $\widetilde{K}$.  
Then, $K$ and $L$ are \emph{$F_n$-similar} if there exist $A_i$ ($1 \leq i \leq n$) such that  
\begin{itemize}
\item $A_i \cap A_j = \emptyset$ $(i \neq j)$, 
\item $\widetilde{L}$ is obtained from $\widetilde{K}$ by forbidden moves at the triangles in any nonempty subfamily of $\{A_i~|~1\leq i \leq n\}$.  
\end{itemize}
In particular, if a virtual knot $K$ and a trivial knot is $F_n$-similar, $K$ is \emph{$F_n$-trivial}.  
\end{definition}
A virtual knot that is $GPV_n$-trivial ($F_n$-trivial, resp.) is given by Lemma~\ref{lemma2} (Lemma~\ref{lemma4}, resp.).  
\section{Main Results}
\begin{theorem}\label{thmGPV} 
For any classical knot $K$, any positive integers $l$, $m$, and $n$ $(m \le n-1)$, and any finite type invariant $v^{\gpv}_m$ of GPV-order $m$, there exist infinitely many of classical knots $K^{\ell}_n$ such that $v^{\gpv}_m(K^{\ell}_n)=v^{\gpv}_m(K)$.
\end{theorem}
\begin{theorem}\label{thmFn}
Let $O$ be a trivial knot.  For any positive integers $m$ and $n$ $(m \le n-1)$, and for any finite type invariant $v^F_m$ of F-order $m$, there exists a nontrivial virtual knot $K_n$ such that $v^F_m (K_n)=v^F_m (O)$.
\end{theorem}

\section{Proofs of Theorem~\ref{thmGPV} and Theorem~\ref{thmFn}}
\subsection{Proof of Theorem~\ref{thmGPV}}
To begin with, we introduce a new kind of crossing, which is called \emph{semi-virtual}.  At a semi-virtual crossing there are still over/under information.  In a diagram, every semi-virtual crossing is shown as a real one, but surrounded by a small circle.   Every semi-virtual crossing is related to the other types of crossings by the following formal relation: 
\begin{align*}
\parbox{30pt}{\includegraphics[width=30pt]{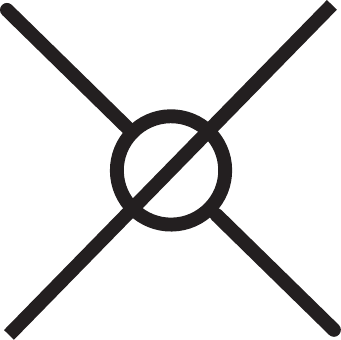}}~
:=~\parbox{30pt}{\includegraphics[width=30pt]{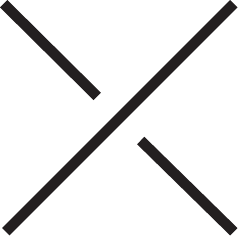}}~
-~\parbox{30pt}{\includegraphics[width=30pt]{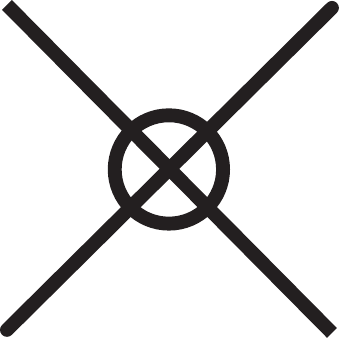}}.  
\end{align*}
Let $A_i=\{c_{i1}, c_{i2}, \ldots, c_{i\alpha(i)}\}$
and let
$
K\left(
    \begin{array}{cccc}
      1 & 2 & \ldots & k \\
      i_1 & i_2 & \ldots & i_k \\
    \end{array}
  \right)
$ be a diagram with semi-virtual crossings by replacing a real crossing \parbox{10pt}{\includegraphics[width=10pt]{crossing1.pdf}} with the virtual crossing \parbox{10pt}{\includegraphics[width=10pt]{virtual2.pdf}} at $
	 c_{11}, \ldots, c_{1i_1-1}, 
	 c_{21}, \ldots, c_{2i_2-1},  
	 \dots, 
	 c_{k1}, \ldots, c_{ki_k-1}
$ and replacing a real crossing \parbox{10pt}{\includegraphics[width=10pt]{crossing1.pdf}} with the semi-virtual crossing \parbox{10pt}{\includegraphics[width=10pt]{semi1.pdf}} at $
	 c_{1i_1}, c_{2i_2}, \dots, 
	 c_{ki_k}
$ (see Example~\ref{K-matrix_eg}).  
Here, note that
\begin{align*}
\parbox{30pt}{\includegraphics[width=30pt]{crossing1.pdf}}~
=~\parbox{30pt}{\includegraphics[width=30pt]{semi1.pdf}}~
+~\parbox{30pt}{\includegraphics[width=30pt]{virtual2.pdf}}.  
\end{align*}
Then, similarly to \cite[Page~289, Lemma~3]{Ohyama2} by Ohyama, it is elementary to see the following: 
\begin{lemma}
If $K$ and $K'$ are $GPV_n$-similar, a finite type invariant $v^{\gpv}_m$ of GPV-order $m (\le n)$ satisfies
\label{lem:v_n-simi}
\begin{align*}
\label{fig:nsimi}
v^{\gpv}_m(K)=v^{\gpv}_m(K') + \sum_{1\leq i_j \leq \alpha(j), 1 \leq j \leq n} v^{\gpv}_m \left( K\left(
    \begin{array}{cccc}
      1 & 2 & \ldots & n \\
      i_1 & i_2 & \ldots & i_n \\
    \end{array}
 \right)\right).
\end{align*}

In particular, if $K$ and $K'$ are $GPV_n$-similar and $v^{\gpv}_i$ is a finite type invariant for GPV-order $m$ $(\le n-1)$, then
\begin{align*}
v^{\gpv}_m(K)&=v^{\gpv}_m(K').
\end{align*}
\end{lemma}
\begin{lemma}\label{lemma2}
For any positive integers $\ell$ and $n$, there exists a classical knot $K_n^{\ell}$ such that  $K_n^{\ell}$ is $GPV_n$-trivial.  
\label{lem:ex_GPV}
\end{lemma}
\begin{proof}
Let $K^\ell_n$ be a knot and $A_j$ ($1 \le j \le i$) a non-empty set of crossings as in Fig.~\ref{dfn_kn}.    
\begin{figure}[htbp]
\centering
{\includegraphics[width=7.5cm,clip]{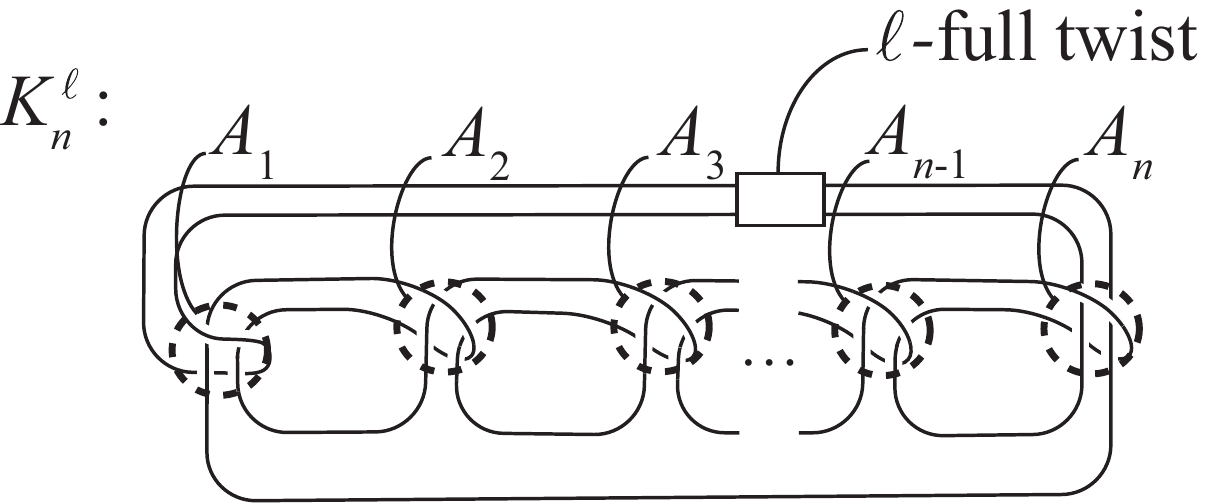}}
%{\includegraphics[width=7.5cm,clip]{ex_F_ntri2.pdf}}
\caption{$K^\ell_n$.}\label{dfn_kn}
\end{figure}
Since $V_{K_n^{\ell}}(t)=(t^2-1)(V_{K_n}(t)-1)\sum_{i=0}^{\ell -1}(t^{-2i})+V_{K_n}(t)\mbox{ and }V_{K_n}(t) \neq 1$  (by Kanenobu \cite{kanenobu1}),
$K_n^{\ell}$ is a nontrivial knot.   Note also that for positive integers $\ell$ and $\ell '$ ($\ell \le \ell '$), $K_n^{\ell}$ and $K_n^{\ell '}$ are not equivalent.
\end{proof}
For two classical knots $L$ and $L'$, $L \sharp L'$ denotes a connected sum of $L$ and $L'$.  Suppose that $K$ is a classical knot.  Recall that, in the proof of Lemma~\ref{lem:ex_GPV}, we give a classical knot $K_n^{\ell}$ concretely.  By definition, it is easy to see that $K$ and $K \sharp K_n^{\ell}$ are $n$-similar.  
Then, by using Lemma \ref{lem:v_n-simi}, $v^{\gpv}_{m} (K)$ $=$ $v^{\gpv}_{m} (K \sharp K_n^{\ell})$.    
$\hfill\Box$
\begin{example}\label{K-matrix_eg}
The key point is that if $K\left(
    \begin{array}{cccc}
      1 & 2 & \ldots & n \\
      i_1 & i_2 & \ldots & i_n \\
    \end{array}
 \right)$ has $n$ columns, then it has $n$ semi-virtual crossings.  For example, a virtual knot diagram of the leftmost figure of Fig.~\ref{K-matrix} is $GPV_2$-trivial by $\{ A_1, A_2 \}$ where $A_1$ $=$ $\{ c_{11}, c_{12} \}$ and $A_2$ $=$ $\{ c_{21}, c_{22} \}$.  By definition, $\alpha(1)$ $=$ $2$ and $\alpha(2)$ $=$ $2$.  Every $K\left(
    \begin{array}{cccc}
      1 & 2 \\
      i_1 & i_2 \\
    \end{array}
 \right)$ is obtained by Fig.~\ref{K-matrix}.  
\begin{figure}[h!]
\includegraphics[width=12cm]{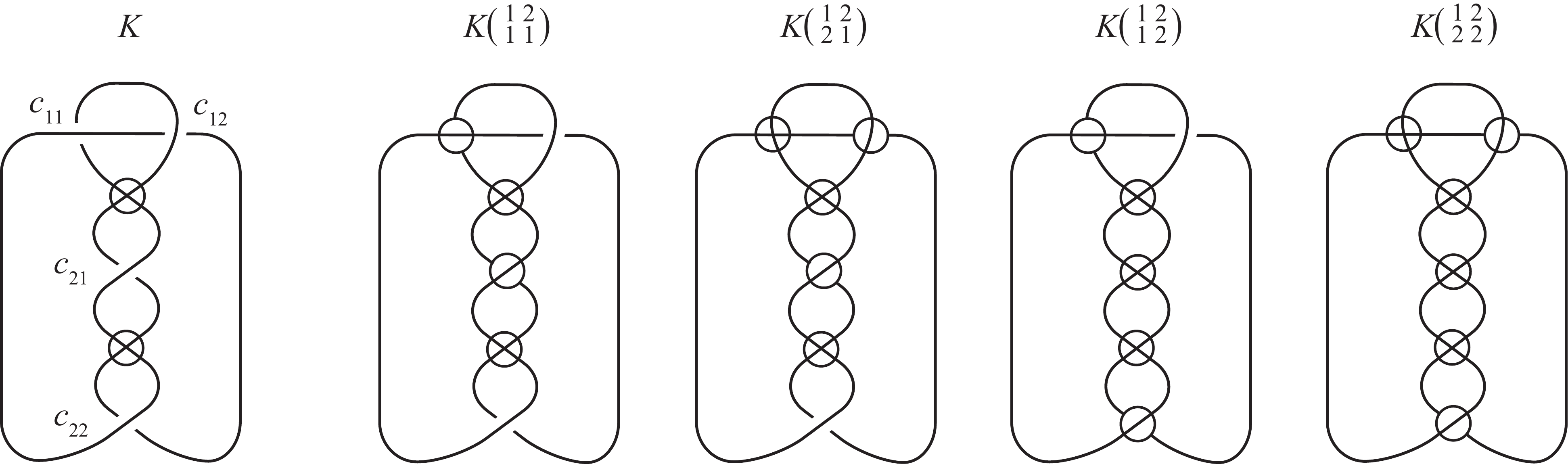}
\caption{$GPV_2$-trivial knot diagrams.}\label{K-matrix}
\end{figure}
\end{example}
\subsection{Proof of Theorem~\ref{thmFn}}\label{secFn}
To begin with, we introduce a new kind of crossing, which is called \emph{semi-triple point}.  A semi-triple point is a sufficiently small disk that consists a single triple point consisting of an over path and a virtual crossing, as \parbox{10pt}{\includegraphics[width=10pt]{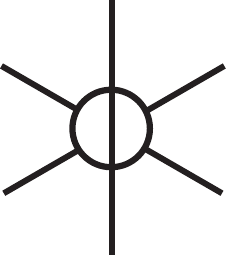}} or \parbox{10pt}{\includegraphics[width=10pt]{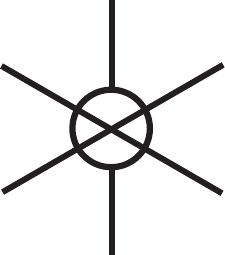}}.  That is, 
at a semi-triple point there is still over/under information.  In a diagram, every semi-triple point is shown as a triple point, but surrounded by a small circle.   Every semi-triple point is related to $\epsilon$- and $(-\epsilon)$-triangles by the following formal relation:
\begin{align*}
\parbox{40pt}{\includegraphics[width=40pt]{triplepoint1.pdf}}~
&:=~\parbox{40pt}{\includegraphics[width=40pt]{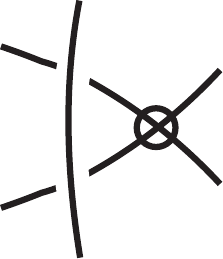}}~
-~\parbox{40pt}{\includegraphics[width=40pt]{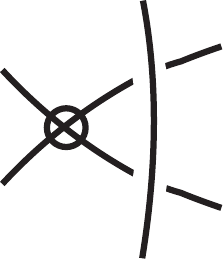}} ,\\
& \qquad {\textrm{positive}} \qquad {\textrm{negative}}\\
\parbox{40pt}{\includegraphics[width=40pt]{triplepoint2.pdf}}~
&:=~\parbox{40pt}{\includegraphics[width=40pt]{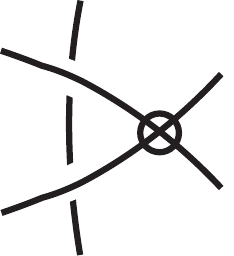}}~
-~\parbox{40pt}{\includegraphics[width=40pt]{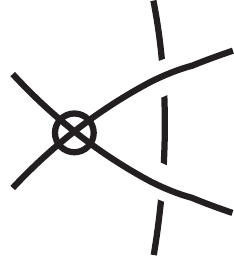}}\\
& \qquad {\textrm{positive}} \qquad {\textrm{negative}}\\
\end{align*}
where the sign of a triangle is defined as in Definition~\ref{dfn_triangle}.  
   
Let $A_i=\{c_{i1}, c_{i2}, \ldots, c_{i\alpha(i)}\}$ consisting of triangles.  Let
$
K\left(
    \begin{array}{cccc}
      1 & 2 & \ldots & k \\
      i_1 & i_2 & \ldots & i_k \\
    \end{array}
  \right)
$ be a diagram with semi-triple points by replacing an $\epsilon$-triangle \parbox{10pt} {\includegraphics[width=10pt]{triangledisk1.pdf}} (\parbox{10pt}{\includegraphics[width=10pt]{triangledisk3.pdf}}, resp.) with the $(-\epsilon)$-triangle \parbox{10pt}{\includegraphics[width=10pt]{triangledisk2.pdf}}  (\parbox{10pt}{\includegraphics[width=10pt]{triangledisk4.pdf}}, resp.)  at $
	 c_{11}, \ldots, c_{1i_1-1}$, 
	 $c_{21}, \ldots, c_{2i_2-1},$ 
	 $\ldots, 
	 c_{k1}, \ldots,$ $c_{ki_k-1}
$ and replacing a triangle \parbox{10pt}{\includegraphics[width=10pt]{triangledisk1.pdf}} (\parbox{10pt}{\includegraphics[width=10pt]{triangledisk3.pdf}}, resp.) with the semi-triple point  \parbox{10pt}{\includegraphics[width=10pt]{triplepoint1.pdf}} (\parbox{10pt}{\includegraphics[width=10pt]{triplepoint2.pdf}}, resp.)
at $
	 c_{1i_1}, c_{2i_2}, \dots, 
	 c_{ki_k}
$.  
Here, note that
\begin{align*}
&\parbox{30pt}{\includegraphics[width=30pt]{triangledisk1.pdf}}~
=~\parbox{30pt}{\includegraphics[width=30pt]{triangledisk2.pdf}}~
+~\parbox{30pt}{\includegraphics[width=30pt]{triplepoint1.pdf}},  \\
& {\text{positive}} \quad
{\text{negative}}
\end{align*}
\begin{align*}
&\parbox{30pt}{\includegraphics[width=30pt]{triangledisk1.pdf}}~
=~\parbox{30pt}{\includegraphics[width=30pt]{triangledisk2.pdf}}~
-~\parbox{30pt}{\includegraphics[width=30pt]{triplepoint1.pdf}},  \\
& {\text{negative}} \quad
{\text{positive}}
\end{align*}
\begin{align*}
\qquad&\parbox{30pt}{\includegraphics[width=30pt]{triangledisk3.pdf}}~
=~\parbox{30pt}{\includegraphics[width=30pt]{triangledisk4.pdf}}~
+~\parbox{30pt}{\includegraphics[width=30pt]{triplepoint2.pdf}},~{\textrm{and}}~  \\
& {\text{positive}} \quad
{\text{negative}}
\end{align*}
\begin{align*}
&\parbox{30pt}{\includegraphics[width=30pt]{triangledisk3.pdf}}~
=~\parbox{30pt}{\includegraphics[width=30pt]{triangledisk4.pdf}}~
-~\parbox{30pt}{\includegraphics[width=30pt]{triplepoint2.pdf}}.  \\
& {\text{negative}} \quad
{\text{positive}}
\end{align*}
That is, every $\epsilon$-triangle equals the sum of $(-\epsilon)$-triangle and $\pm$ semi-triple point where the sign $\pm$ coincides with $\epsilon$.  Then, in the rest of the paper, for a given triangle, the induced sign, $+$ or $-$, which is the coefficient of a semi-triple point, is called \emph{the sign of a semi-triple point}.  For $c_{li_l}$ as above, the sign of a semi-triple point is denoted by $\epsilon_{li_l}$.  

Then, similarly to \cite[Page~289, Lemma~3]{Ohyama2} by Ohyama, it is elementary to see the following:
\begin{lemma}
If $K$ and $K'$ are $F_n$-similar, a finite type invariant $v^F_m$ of F-order $m (\le n)$ satisfies
\label{lem:v_Fn-simi}
\begin{align*}
v^F_m(K)=v^F_m(K') + \sum_{1\leq i_j \leq \alpha(j), 1 \leq j \leq n} \varepsilon_{1i_1}\varepsilon_{2i_2}\cdots \varepsilon_{ni_n} v^F_m \left( K\left(
    \begin{array}{cccc}
      1 & 2 & \ldots & n \\
      i_1 & i_2 & \ldots & i_n \\
    \end{array}
 \right)\right)
\end{align*}
where $\varepsilon_*$ is the sign of a semi-triple point.  

In particular, 
if $K$ and $K'$ are $F_n$-similar and $v^F_m$ is a finite type invariant of F-order $m$ $(\le n-1)$, then
\begin{align*}
v^F_m(K)&=v^F_m(K').
\end{align*}
\end{lemma}
\begin{notation}
Every braid appearing in this paper is a pure $4$-braid.  Thus, it is called a \emph{braid} simply.  
For a given braid $b$, the virtual knot diagram $\hat{b}$ obtained from $b$ by Fig.~\ref{closure} is called the \emph{closure} of the braid. 
\begin{figure}[htbp]
\includegraphics[width=4cm]{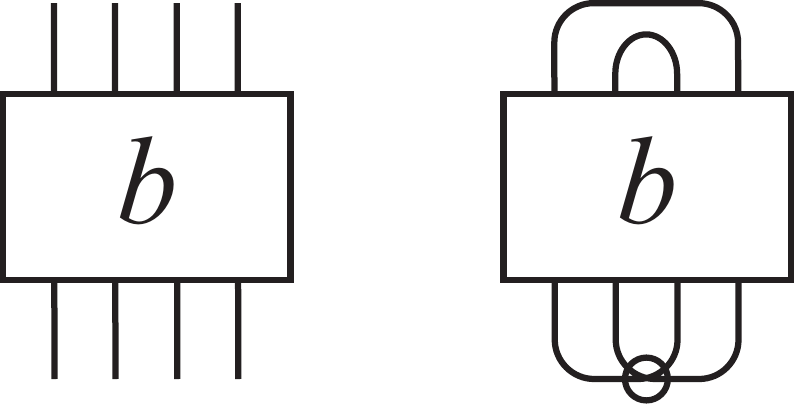}
\caption{A braid $b$ (left) and its closure $\hat{b}$ (right).}\label{closure}
\end{figure}  
\end{notation}
\begin{lemma}\label{lemma4}
For any positive integer $n$, there exists $K_n$ such that $K_n$ is $F_n$-trivial.
\label{lem:ex_Fn}
\end{lemma}

\begin{proof}
Let
\begin{align*}
A =\parbox{30pt}{\includegraphics[width=30pt]{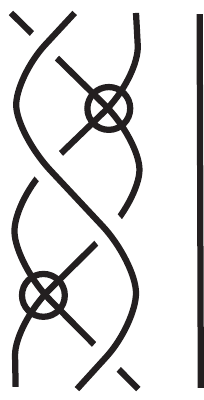}}~, 
~A^{-1} =~\parbox{30pt}
{\includegraphics[width=30pt]{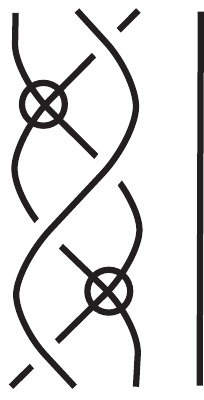}}~,~ B =\parbox{30pt}{\includegraphics[width=30pt]{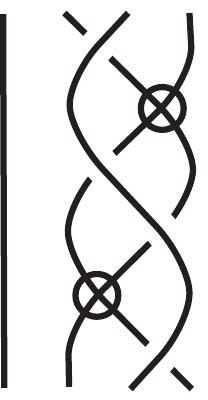}}~, ~{\textrm{and}}~
B^{-1} =~\parbox{30pt}{\includegraphics[width=30pt]{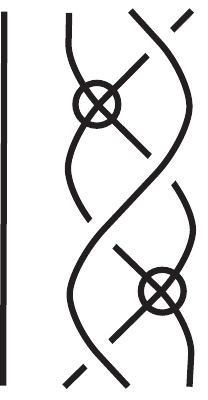}}. 
\end{align*}

Let $b(k)$ be a braid with a non-empty set $A_l$ ($1 \le l \le k$) consisting of disjoint triangles, $b(k)^{-1}$ its inverse, $\hat{b}(k)$ its closure, $[\cdot, \cdot]$ a commutator, and $D(\hat{b}(k))$ a virtual knot diagram defined as follows: 
$b(1)=A,$ $b(2)=[B,b(1)]$, $b(4u-1)=[B,b(4u-2)]$,
$b(4u)=[A,b(4u-1)]$, $b(4u+1)=[A,b(4u)]$, $b(4u+2)=[B,b(4u+1)]$ $(u \geq 1).$  Note that the definition of $b(k)^{-1}$ having $A_l$ ($1 \le l \le k$) is defined as the mirror image of $b(k)$ with $A_l$ ($1 \le l \le k$).   
  
\begin{align*}
\begin{array}{cccccc}
{\hat{b}(1)}&{\hat{b}(2)}&{\hat{b}(4u-1)}\\
{\parbox{40pt}{\includegraphics[width=40pt]{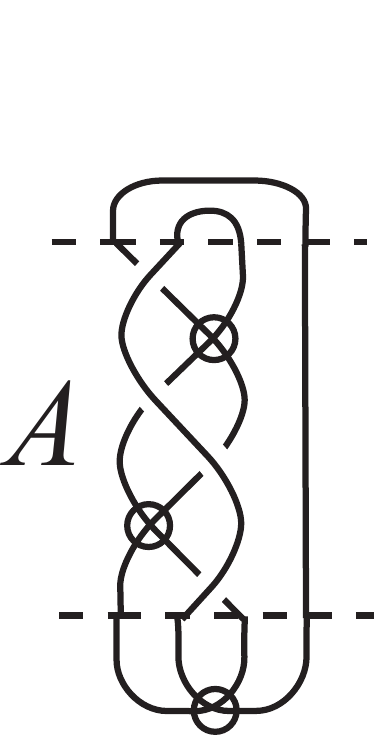}}}
{\parbox{40pt}{\includegraphics[width=40pt]{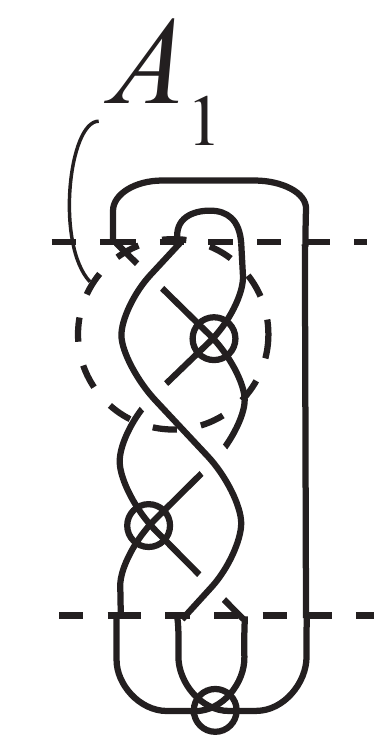}}}
&{\parbox{40pt}{\includegraphics[width=40pt]{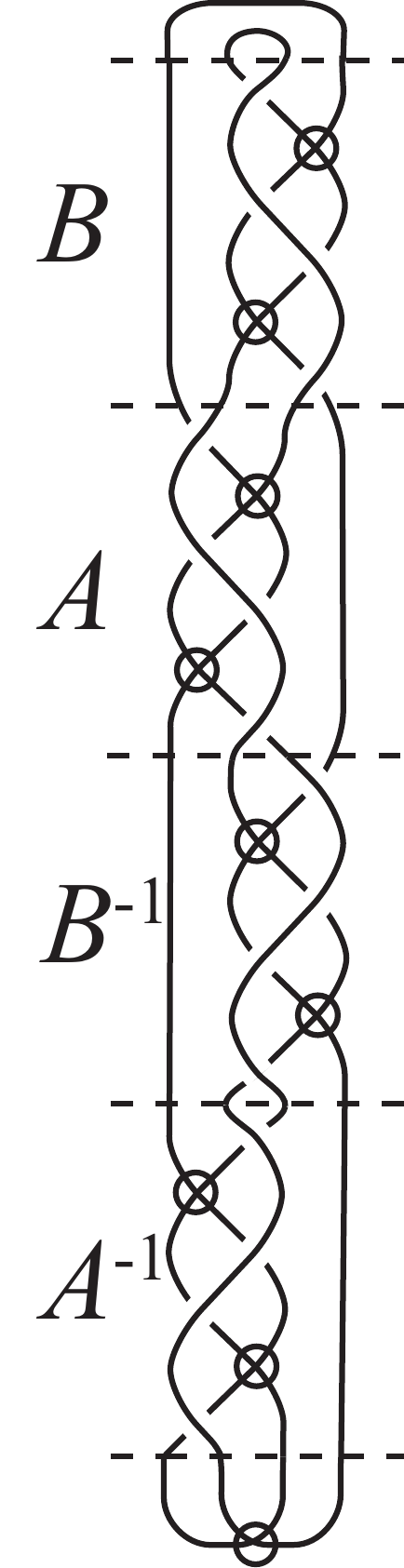}}}
{\parbox{40pt}{\includegraphics[width=40pt]{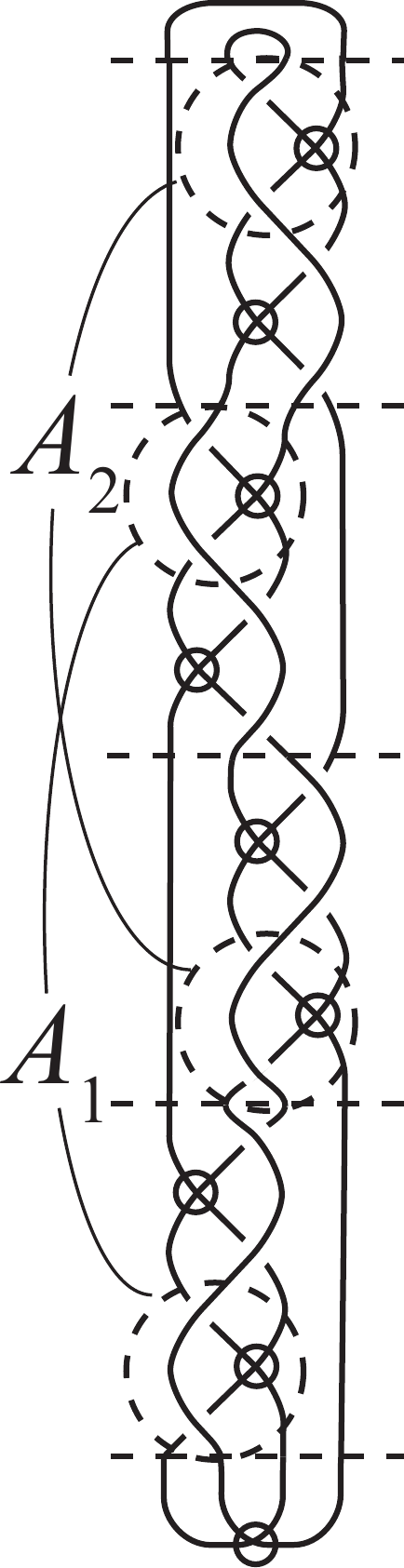}}}
&{\parbox{40pt}{\includegraphics[width=40pt]{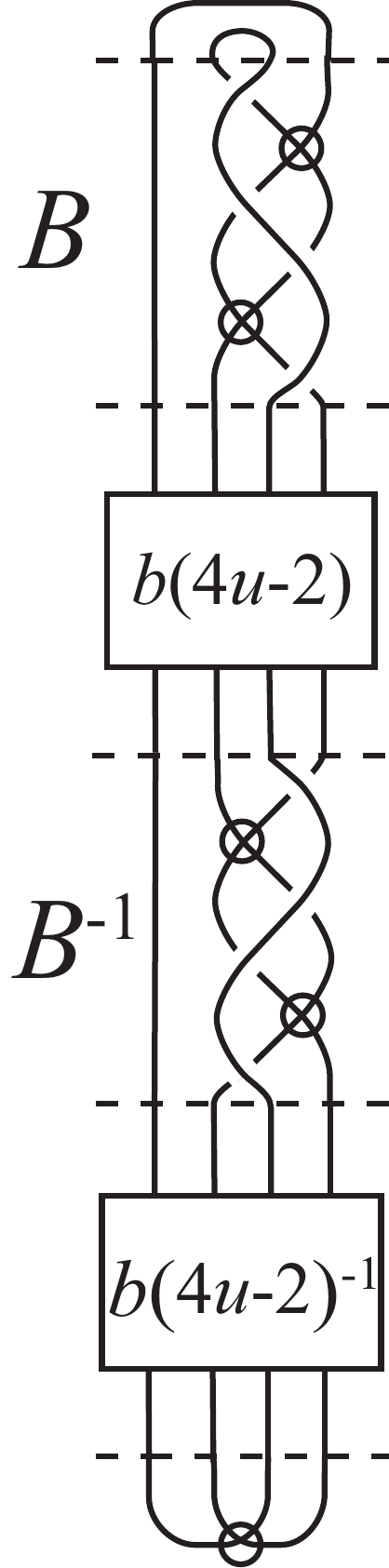}}}
{\parbox{40pt}{\includegraphics[width=40pt]{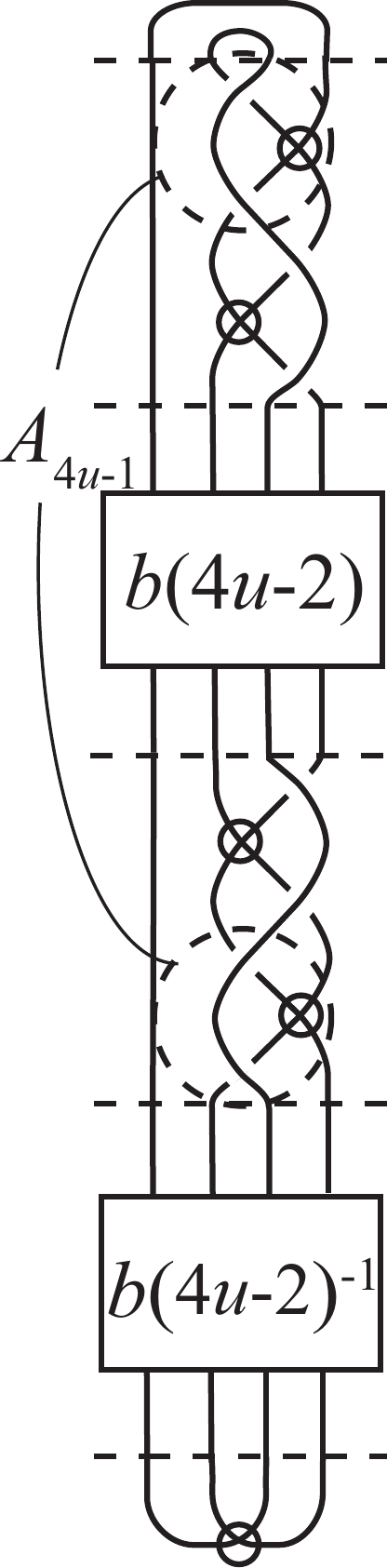}}}\\
\\
{\hat{b}(4u)}& {\hat{b}(4u+1)} &{\hat{b}(4u+2)}\\
{\parbox{45pt}{\includegraphics[width=45pt]{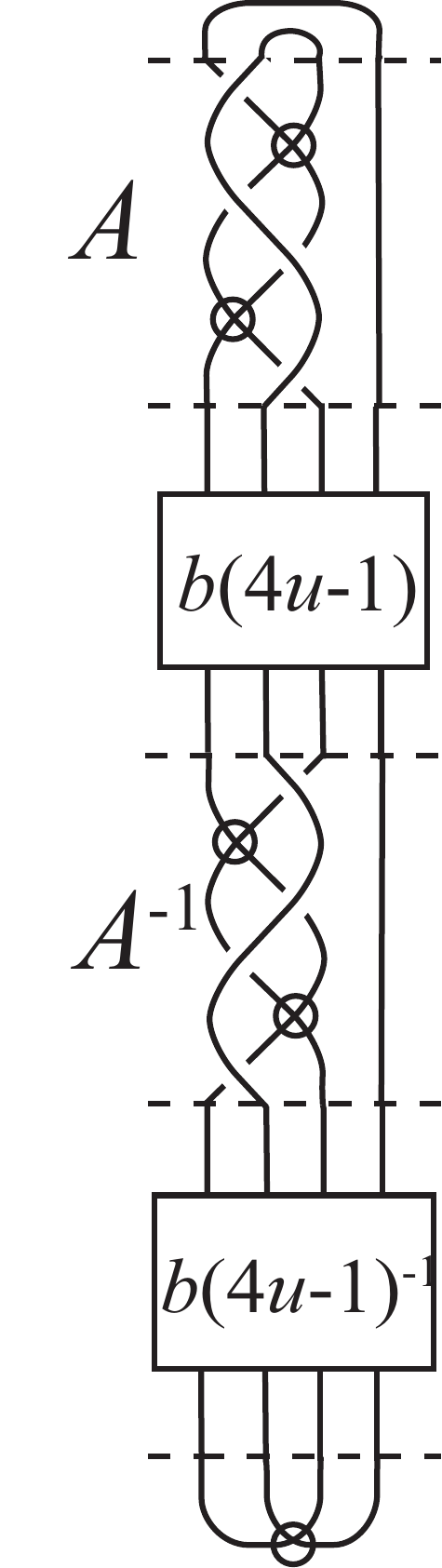}}}
{\parbox{45pt}{\includegraphics[width=45pt]{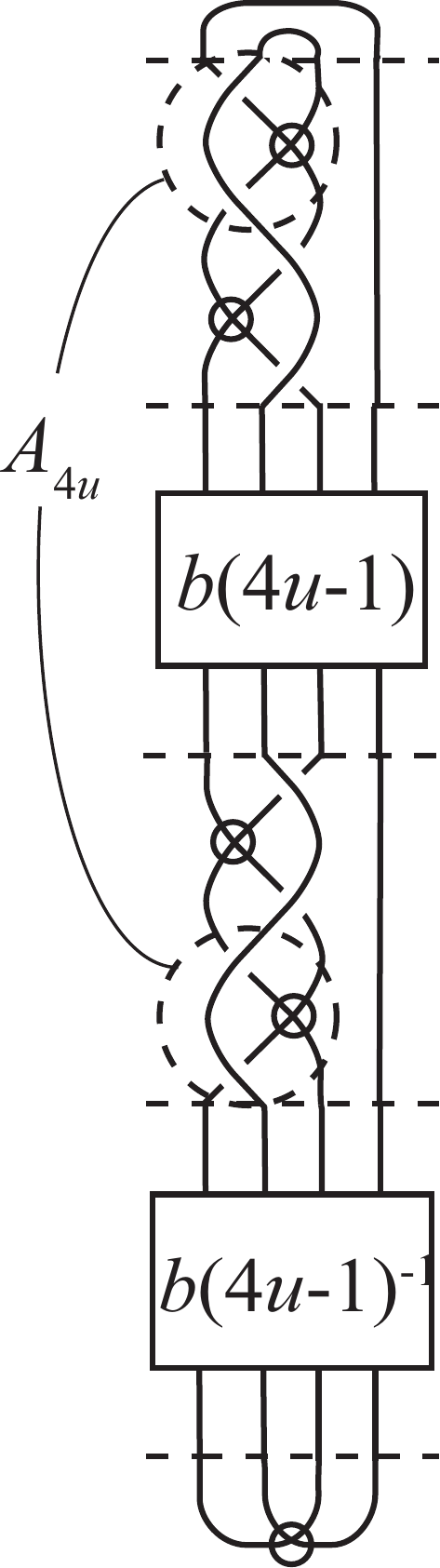}}}
&{\parbox{45pt}{\includegraphics[width=45pt]{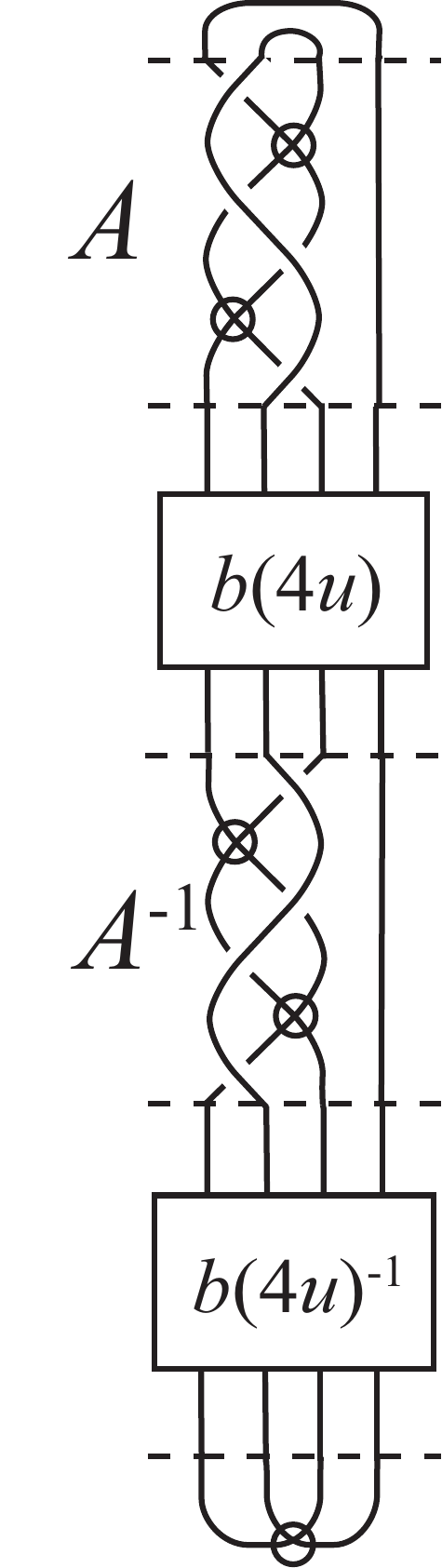}}}
{\parbox{45pt}{\includegraphics[width=45pt]{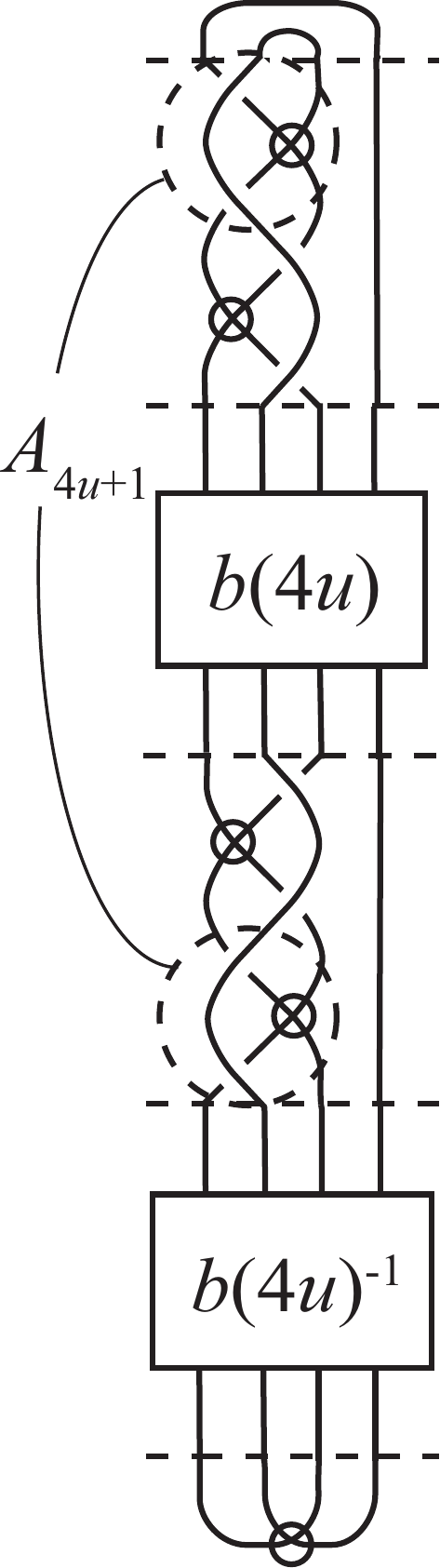}}}
&{\parbox{40pt}{\includegraphics[width=40pt]{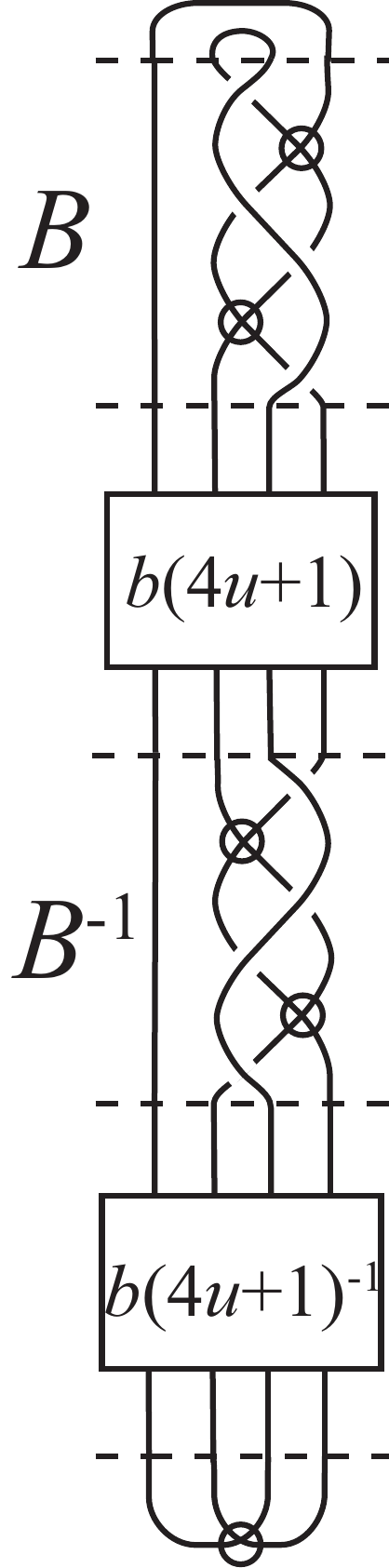}}}
{\parbox{40pt}{\includegraphics[width=40pt]{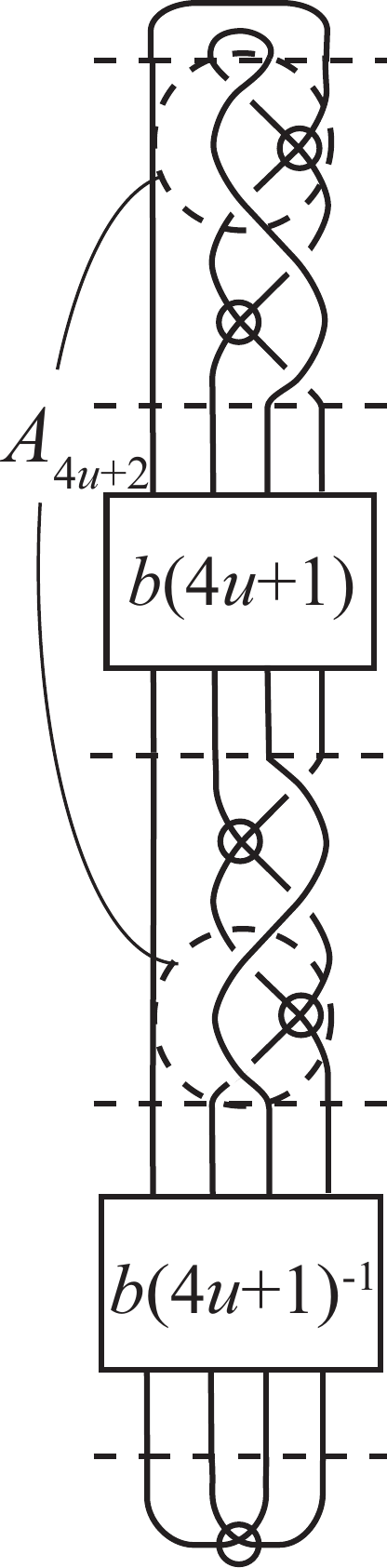}}}
\end{array}
\end{align*}

Let $\kh$ be a Khovanov homology by Manturov \cite{manturov} (for our notation, see Section~\ref{secManturov}).  Then, since for every $k \in \mathbb N$, there exist $i, j$ $( j \neq \pm 1)$ such that $\rank\, \left( \kh^{i, j} (\hat{b}(k)) \right) \geq 1$ by Section~\ref{secComputation}.  
Note that $\kh^{0, -1}({\textrm{unknot}})$ $=$ ${\mathbb{Z}}_2$, $\kh^{0, 1}({\textrm{unknot}})$ $=$ ${\mathbb{Z}}_2$, and for $i, j$ ($|j| \neq 1$), $\kh^{i, j}({\textrm{unknot}})$ $=$ $0$.  Thus, $\hat{b}(k)$ is a nontrivial virtual knot.    
\end{proof}

By Lemma~\ref{lem:v_Fn-simi} and Lemma~\ref{lem:ex_Fn}, we have  Theorem~\ref{thmFn}.
$\hfill\Box$

\section{On a computation of $\kh(\hat{b}(k))$.}\label{secComputation}

In the following, for every $\hat{b}(k)$, we show that there exist $i, j (|j| \neq 1)$ such that $\kh^{i, j}(\hat{b}(k))$ $\neq 0$.  We use definitions and notations in Section~\ref{secViro} for the Khovanov homology.  
Let $D$ $=$ $D(\hat{b}(k))$ as in Section~\ref{secFn} (see Lemma~\ref{lemma4}).  
We define the set $\mathcal{S}_{i, j_0}$ by
\[
{\mathcal{S}}_{i, j_0} = \{ S : {\textrm{enhanced state}}~|~i-1 \le i(S) \le i+1, j(S)=j_0,~{\textrm{every circle has the label}}~1 \}.  
\]
By definition, $\mathcal{S}_{i, j_0}$ is a generating set of the direct sum of chain groups $C^{i-1, j_0}(D)$ $\oplus C^{i, j_0}(D)$ $\oplus C^{i+1, j_0}(D)$.   
For a given virtual knot diagram, the arrangement of disjoint circles on a plane by smoothing along every marker of each real crossing is called a \emph{state}.  A state with labels by assigning a label $x$ or $1$ for every circle in a state is called an \emph{enhanced state}.    
For an enhanced state $S$ (a state $s$, resp.), let $|s|$ $=$ the number of circles in $s$.  
\begin{lemma}\label{lemma5}
Let $S \in \mathcal{S}_{i, j_0}$.  Let $s$ be the state obtained from $S$ by ignoring labels $x, 1$ satisfying the following conditions $(\ref{1})$ and $(\ref{2})$: 
\begin{enumerate}
\item  For every state $t$ obtained from $s$ by replacing a positive marker with the negative maker, $|t| = |s|$.  
\label{1}
\item For every state $t$ obtained from $s$ by replacing a negative marker with the positive maker, $|t| \le |s|$.
\label{2}
\end{enumerate}
Then, $C^{i-1, j_0}(D)$ $\oplus C^{i, j_0}(D)$ $\oplus C^{i+1, j_0}(D)$ is generated by $\{S\}$, i.e., 
\[
\mathcal{S}_{i, j_0} = \{ S : {\textrm{enhanced state}}~|~ i-1 \le i(S) \le i+1, j(S)=j_0 \} = \{ S \}.  
\]
\end{lemma}
\begin{proof}
We check the conditions (\ref{1}) and (\ref{2}), respectively.  
\begin{enumerate}
\item By the definition of Khovanov homology in the $\mathbb{Z}_2$ coefficient by Manturov.  Suppose that $S \in \mathcal{S}_{i, j_0}$.  The condition~(\ref{1}) implies that replacing a positive marker with the negative maker does not change the number of component.  
By definition, the differential sends $S$ to $0$.  
\item If we replace a negative marker on $S$ by a positive marker, $i$ decreases by $1$.  Suppose that $|t| \le |s|$.  Then, for every enhanced state $T$ obtained from $t$ with labels, $j(T) < j_0$.  Then, $T \notin \mathcal{S}_{i, j_0}$.   
\end{enumerate}
Therefore, (\ref{1}) and (\ref{2}) imply  
\[
\mathcal{S}_{i, j_0} = \{ S : {\textrm{enhanced state}}~|~i-1 \le i(S) \le i+1, j(S)=j_0 \} = \{ S \}.  
\]
\end{proof}
To prove the existence of $i, j (|j| \neq 1)$ such that $\kh^{i, j}(\hat{b}(k))$ $\neq 0$ for every $\hat{b}(k)$, we find an enhanced state satisfying the conditions (\ref{1}) and (\ref{2}) of Lemma~\ref{lemma5}, which implies
\[
0 \to \mathcal{S}_{i, j_0} \to 0.  
\]
We will find such an enhanced state obtained from $\hat{b}(k)$.  In order to define such enhanced states, we prepare Notation~\ref{notationS}.
\begin{notation}\label{notationS}
For braids $A$ and $B$, for every crossing, we put on a negative marker.  
For braids $A^{-1}$ and $B^{-1}$, if we apply the smoothing along a positive (negative, resp.) marker at a real crossing, then, by a short dotted arc (solid arc, resp.), we indicate that we select the positive (negative, resp.) marker, as shown in Fig.~\ref{dotted}.  The symbol $\delta$ indicates a simple circle that simplifies presentations of enhanced states as in Fig.~\ref{defS}.  
In the rest of this paper, every label on each circle is $1$, and thus, the indication of the labels are omitted. 
\begin{figure}[htbp]
\includegraphics[width=10cm]{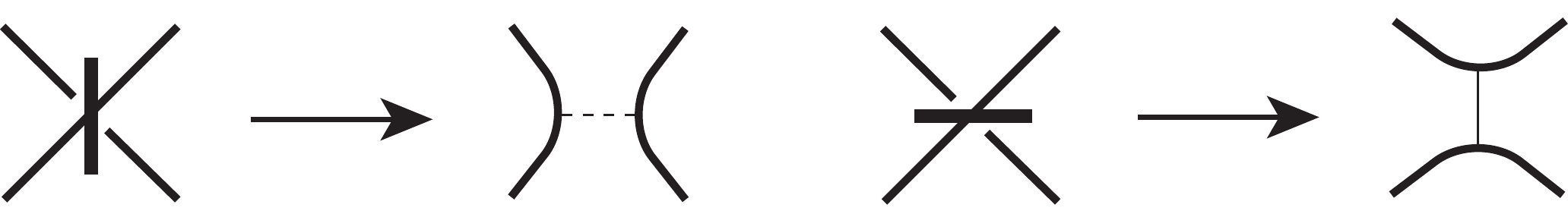}
\caption{a smoothing along a positive marker (left) and a smoothing along a negative marker (right).}\label{dotted}
\end{figure} 
We define an enhanced state $S(\hat{b}(k))$ and $S(\hat{b}(k)^{-1})$ ($2 \le k$) by Fig.~\ref{defS} recursively.      
\end{notation}

\begin{figure}[htbp]
\begin{tabular}{|c|c|c|} \hline
$S(\hat{b}(1))$ & $S(\hat{b}(2))$ & $S(\hat{b}(2)^{-1})$
\\ \hline
\includegraphics[width=2cm]{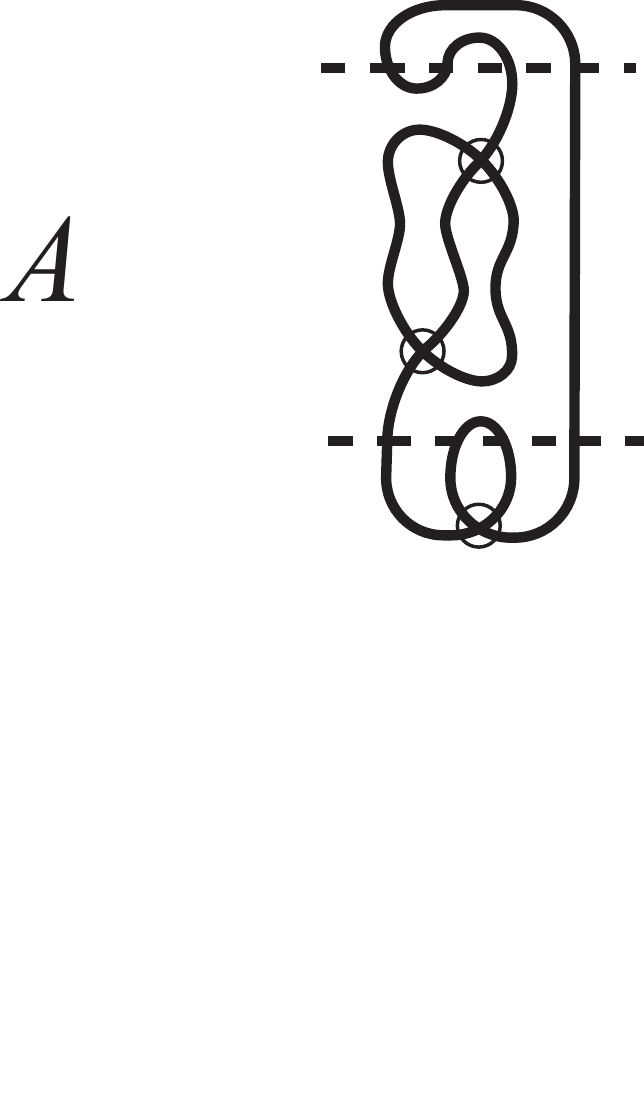}
&
\includegraphics[width=2cm]{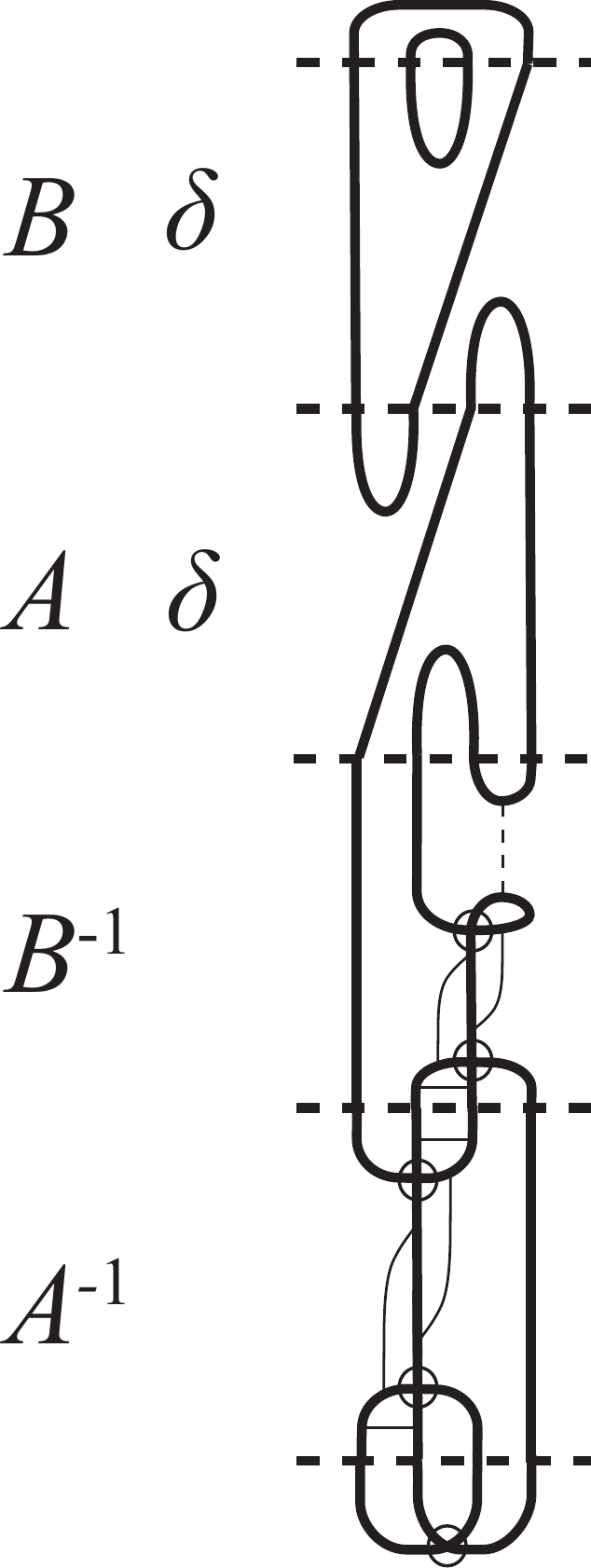}
& 
\includegraphics[width=2cm]{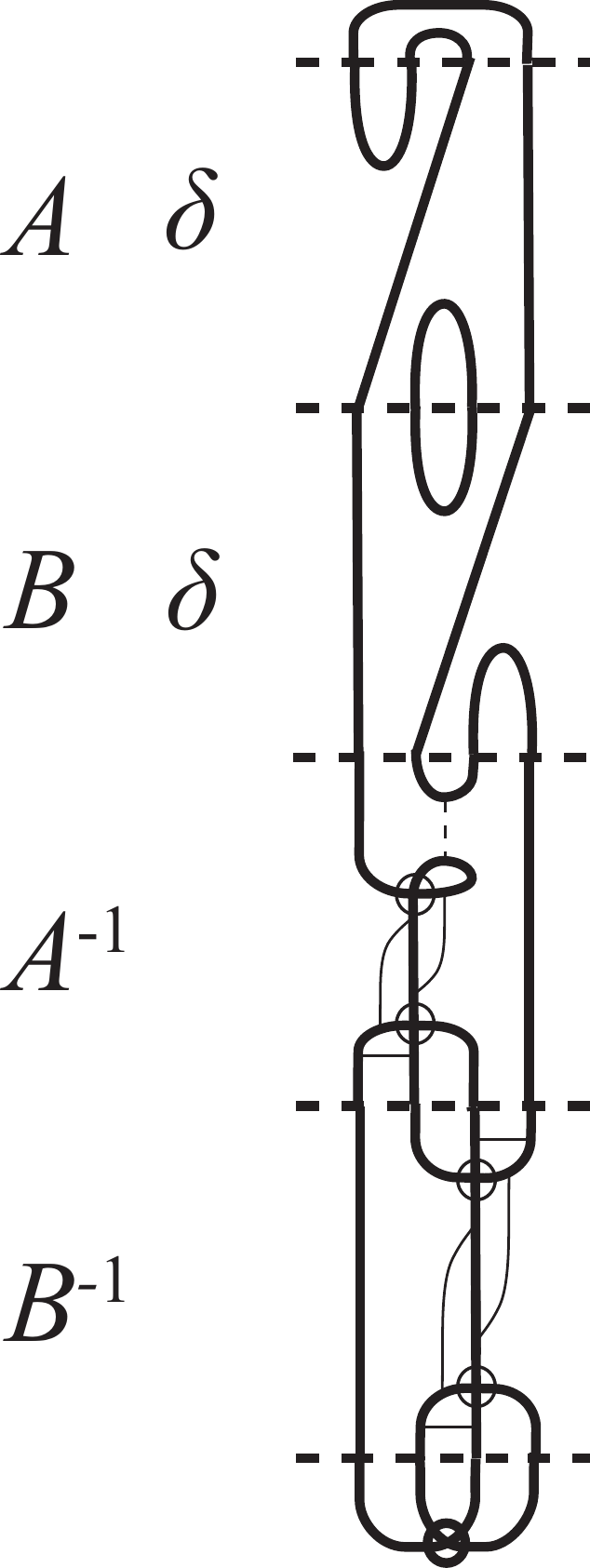}  \\ 
\hline
\end{tabular}
\caption{$S(\hat{b}(1))$, $S(\hat{b}(2))$, and $S(\hat{b}(2)^{-1})$. The symbol $\delta$ indicates a simple circle. }

\vspace{0.5cm}
\begin{tabular}{|c|c|c|c|} \hline
$S(\hat{b}(4u-1))$ & $S(\hat{b}(4u-1)^{-1})$ & $S(\hat{b}(4u))$ & $S(\hat{b}(4u)^{-1})$
\\ \hline
\includegraphics[width=2cm]{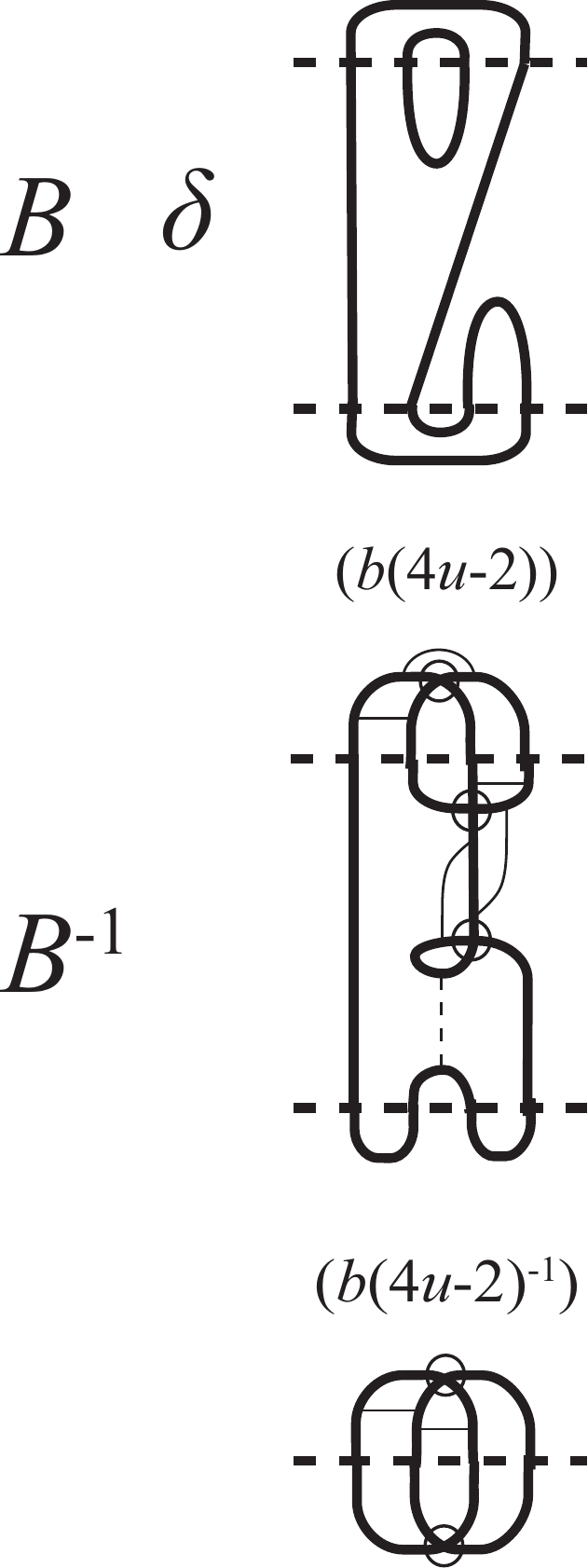}
&
\includegraphics[width=2cm]{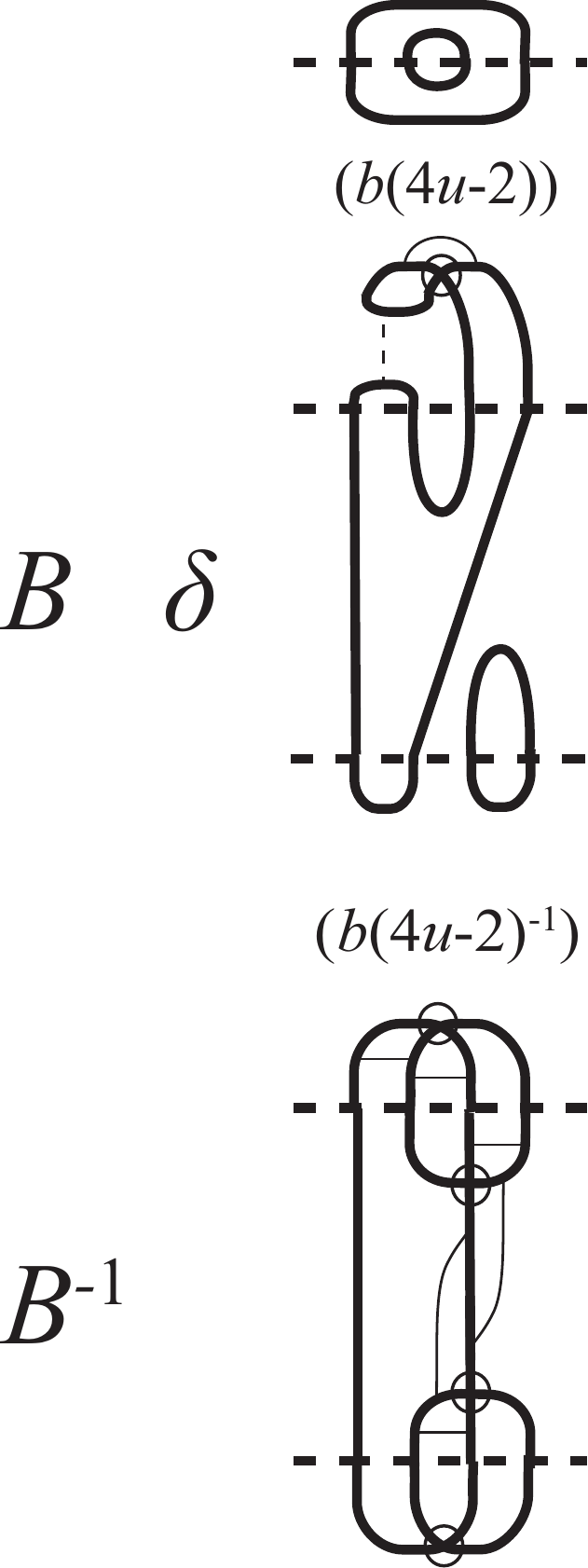}
&
\includegraphics[width=2cm]{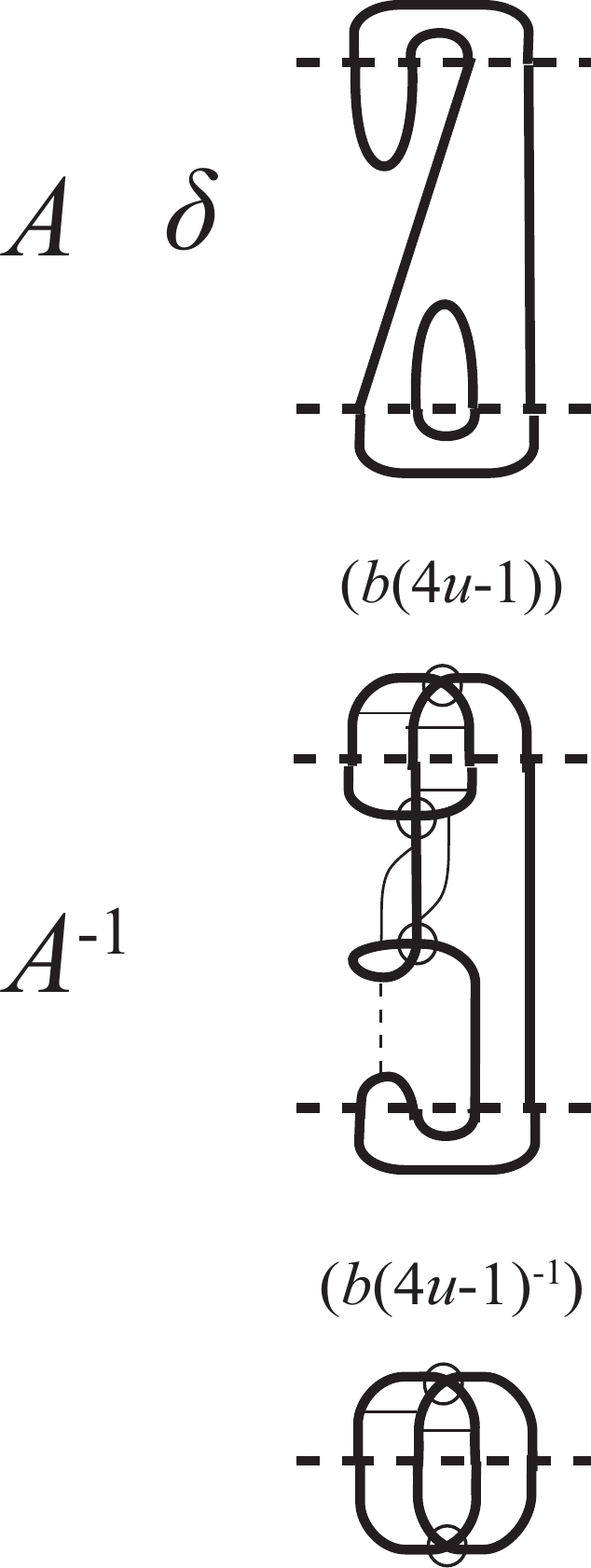}
& 
\includegraphics[width=2cm]{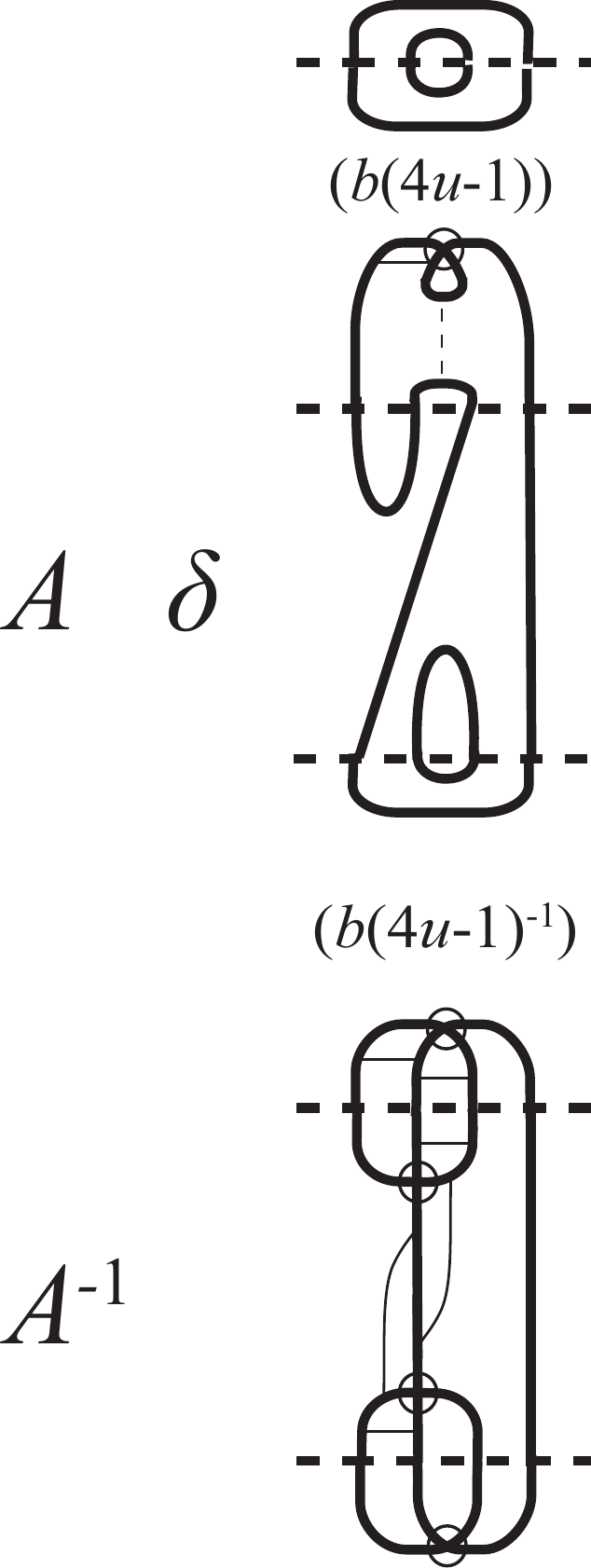}  \\ 
\hline
$S(\hat{b}(4u+1))$ & $S(\hat{b}(4u+1)^{-1})$ & $S(\hat{b}(4u+2))$ & $S(\hat{b}(4u+2)^{-1})$\\ \hline
\includegraphics[width=2cm]{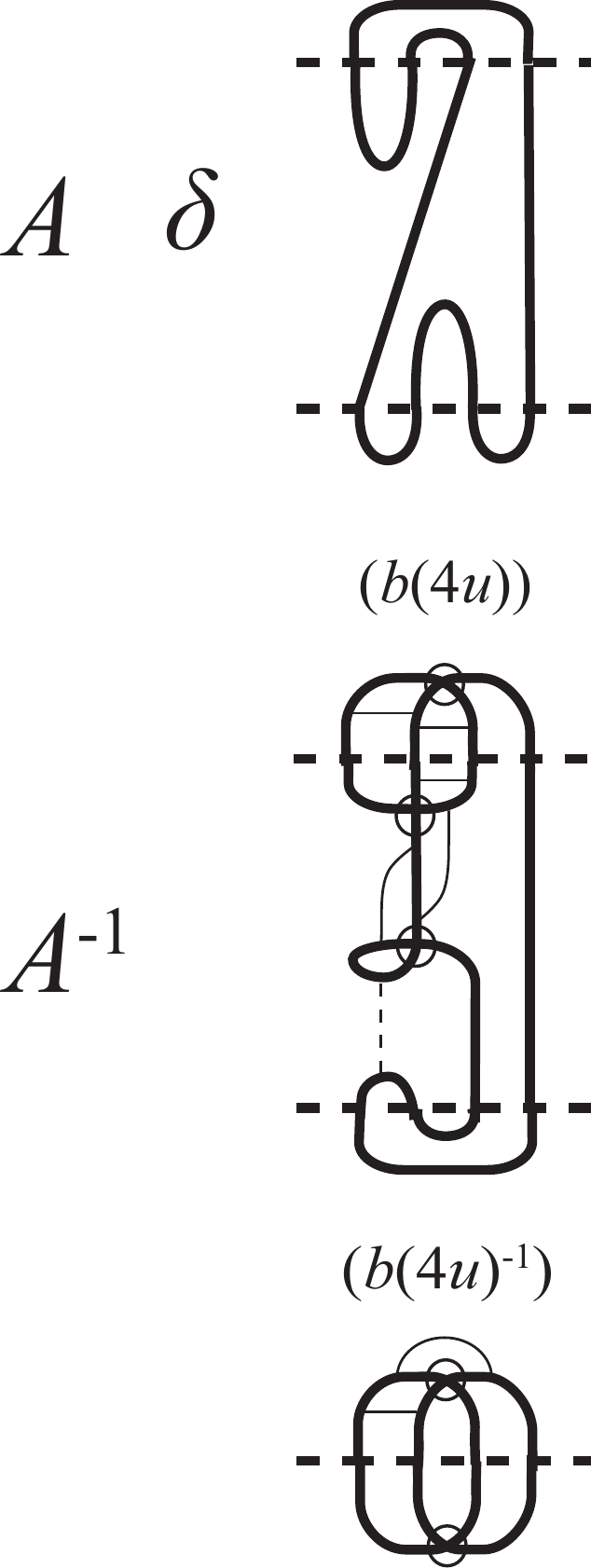}
&
\includegraphics[width=2cm]{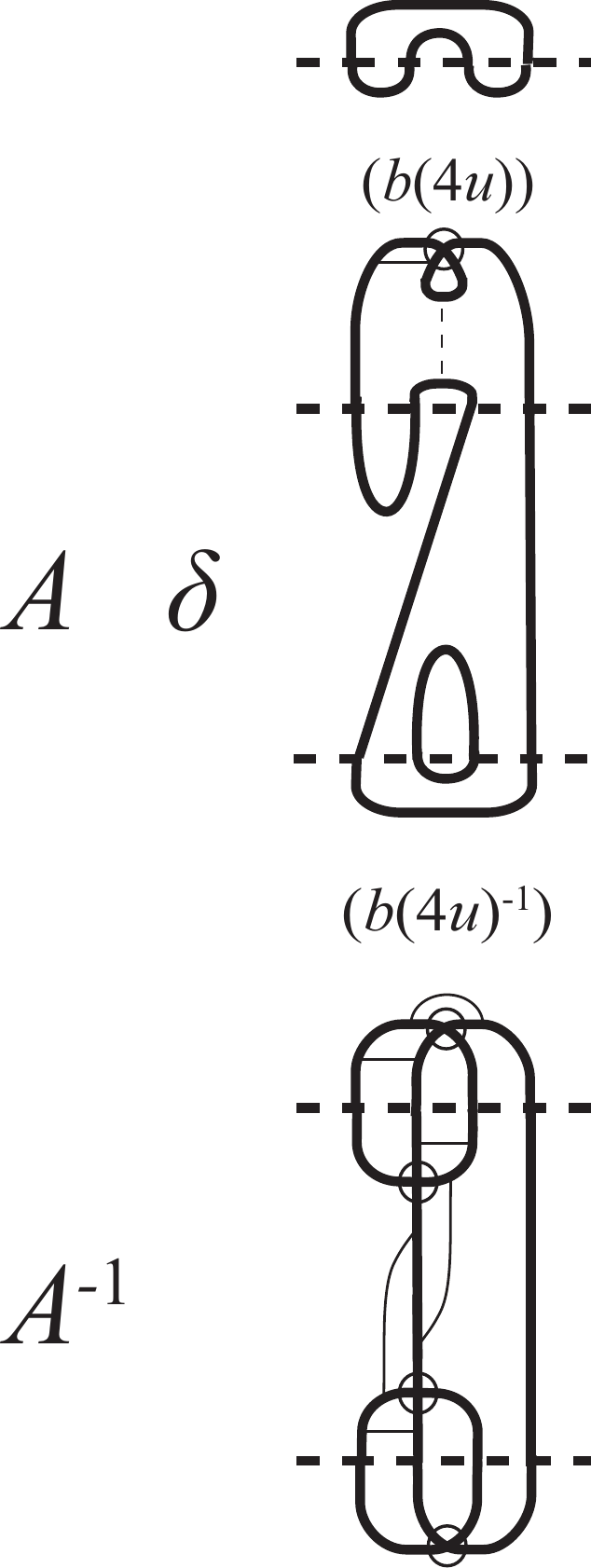}
&
\includegraphics[width=2cm]{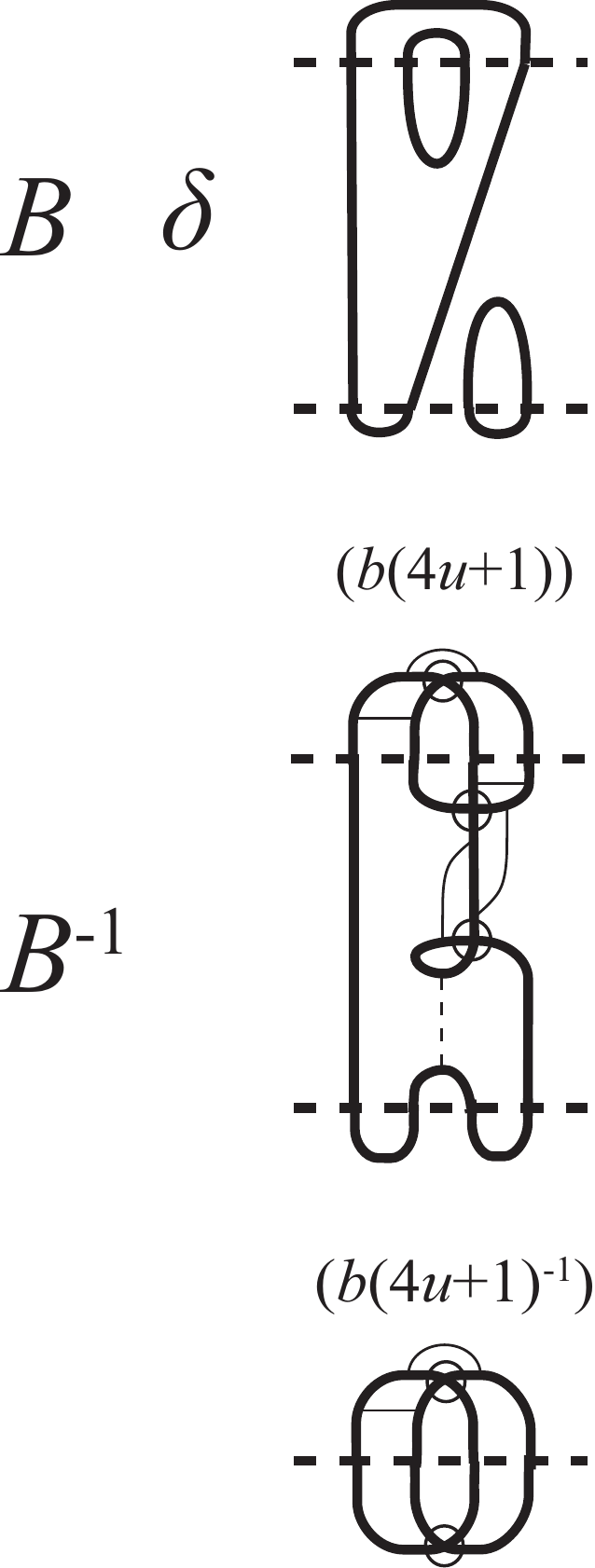}
&
\includegraphics[width=2cm]{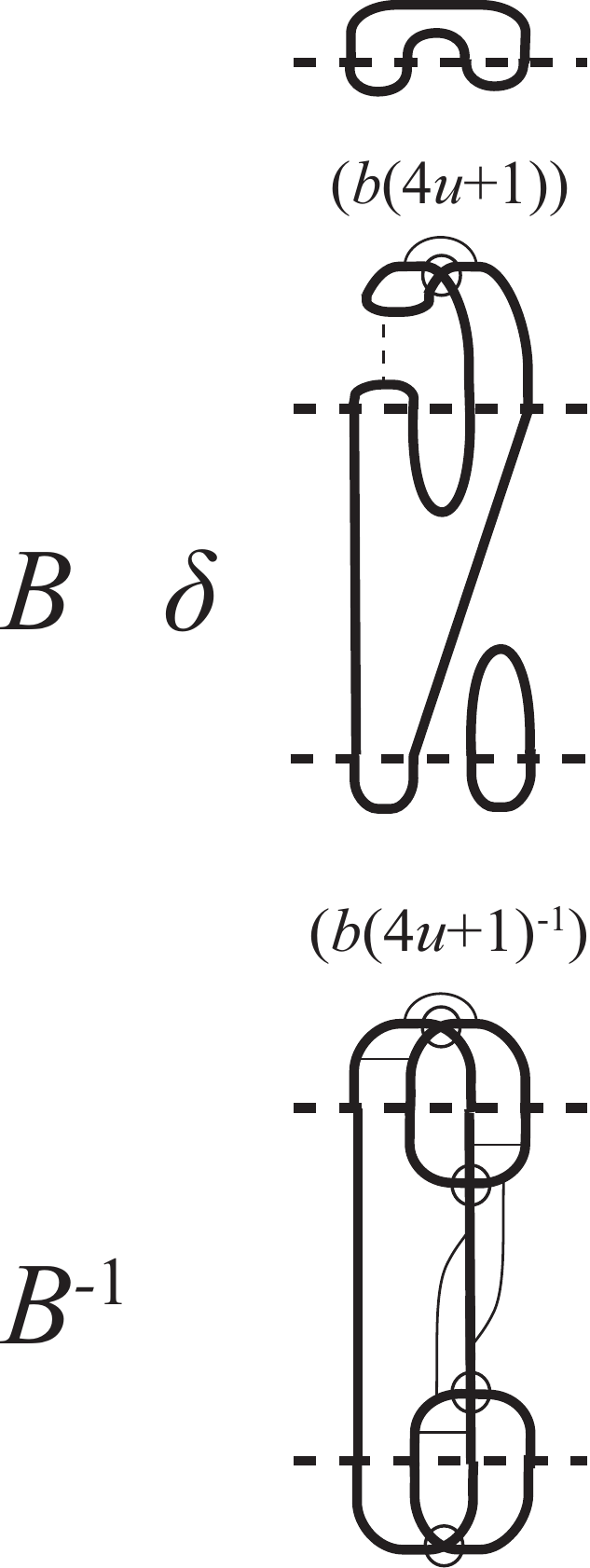}\\ \hline
\end{tabular}
\caption{$S(\hat{b}(k))$ and $S(\hat{b}(k)^{-1})$ $(k \ge 3)$.  $(b(k))$ ($(b(k)^{-1})$, resp.) indicates the part of $S(\hat{b}(k))$ ($S(\hat{b}(k)^{-1})$, resp.) between two adjoining dotted segments.  The symbol $\delta$ indicates a simple circle.}\label{defS}
\end{figure}  
By using Fig.~\ref{defS}, it is elementary to check $S(\hat{b}(k))$ ($1 \le k \le 6$) satisfies the conditions (\ref{1}) and (\ref{2}) by the direct computation for every crossing.  
By using the recursive formulae of definitons, it is easy to see that  $S(\hat{b}(k))$ ($7 \le k$) satisfies the conditions (\ref{1}) and (\ref{2}).  
Finally, we show that $|j(S(k))| \neq 1$ in the following.  Note that, by definition, $w(\hat{b}(k))$ $=$ $0$.  Note also that by definition, $\sigma(S(\hat{b}(k))) < 0$.  Then, for $S(\hat{b}(k))$ ($k \ge 2$), $j(S(\hat{b}(k)))$ $=$ $-\frac{1}{2} \sigma(S(\hat{b}(k))) +\tau(S(\hat{b}(k)))$ $>$ $\tau(S(\hat{b}(k)))$.  Here, it is easy to see that $\tau(S(\hat{b}(k)))$ $>2$ by Fig.~\ref{defS} (see the number of $\delta$).     
Thus, $S(\hat{b}(k))$ ($2 \le k$) satisfies $|j(S(\hat{b}(k)))| \neq 1$ and the conditions (\ref{1}) and (\ref{2}).  By Lemma~\ref{lemma5}, for every $k \in \mathbb{N}$ ($k \ge 2$), there exist $i$ and $j$ ($|j| \neq 1$) such that $\kh^{i, j}(\hat{b}(k))$ $\neq 0$.   
$\hfill \Box$

%\appendix
\section{Definition of Khovanov homology for virtual knots}
In this section, we recall a definition of a Khovanov homology for virtual knots.  It is worth giving the following short review of Khovanov homologies of Viro \cite{Vi} and Manturov \cite{manturov} because readers easily check our notation.  
\subsection{A definition of the Khovanov homology by Viro for knots}\label{secViro}
The Khovanov homology of the Jones polynomial is introduced by Khovanov \cite{khovanov}.  There are at least two famous redefinitions of the Khovanov homology.  Here, we give a brief review of the definition of Viro \cite{Vi}.  Before starting with the review, we define a \emph{link} and a \emph{link diagram}.    

\begin{definition}[link, link diagram]
A \emph{link} of $k$ components is the image of a smooth embedding of the disjoint union of $k$ circles into $\mathbb{R}^3$.  In particular, a knot is a link of $1$ component.     Two links $L$ and $L'$ are \emph{isotopic} if there exists a smooth family of homeomorphisms $h_t : \mathbb{R}^3$ $\to$ $\mathbb{R}^3$ for $t \in [0, 1]$ such that $h_0$ is the identity map of $\mathbb{R}^3$ and $h_1 (L)$ $=$ $L'$.  Then, such a family of $h_t$ is called an \emph{isotopy} of $\mathbb{R}^3$.   
\end{definition}
For an unoriented link diagram, we recall the definition of the Kauffman bracket $\langle \cdot \rangle$ that is the map from the set of unoriented link diagrams to $\mathbb{Z}[A, A^{-1}]$.    For a link diagram $D$, let $D_0$ ($D_{\infty}$, resp.) be link diagrams obtained by replacing a disk ($\supset$ a single crossing) with the other disk $d_0$ (res.~$d_{\infty}$), as shown in Fig.~\ref{kh1}.
\begin{figure}[htbp]
\includegraphics[width=3cm]{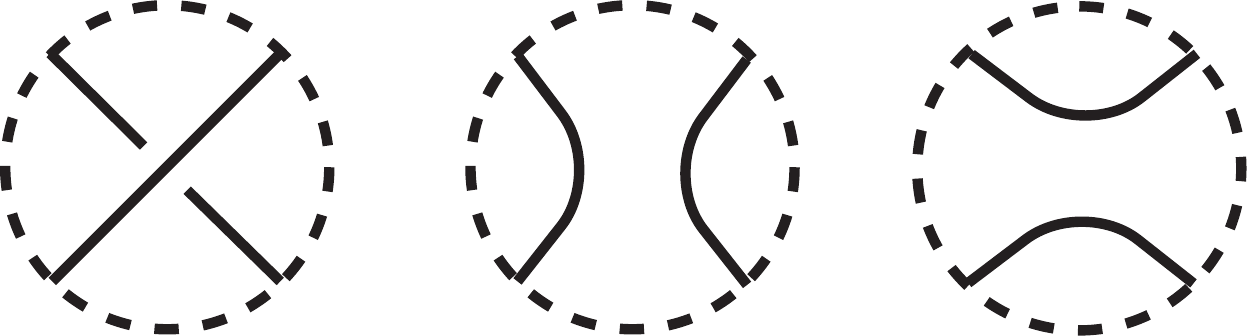}
\caption{(a) a disk ($\supset$ a single crossing), (b) $d_0$, (c) $d_{\infty}$.}\label{kh1}
\end{figure}
Then, it is defined by the following conditions:
\begin{enumerate}
\item $\langle \emptyset \rangle$ $=$ $1$.  
\item $\langle D \amalg {\textrm{knot diagram with no crossings}} \rangle$ $=$ $(-A^2 -A^{-2}) \langle D \rangle$, where $\amalg$ stands for disjoint sum.  
\item $\langle D \rangle$ $=$ $A \langle D_0 \rangle$ $+$ $A^{-1} \langle D_{\infty} \rangle$.  
\end{enumerate}
For each crossing of $D$, the replacing $D$ with $D_0$ or $D_{\infty}$ is called a \emph{smoothing}.  
Note that a smoothing from $D$ to $D_0$ ($D_{\infty}$, resp.)  corresponding to $A$ ($A^{-1}$, resp.) and then we call the smoothing corresponding to $A$ a \emph{positive} (\emph{negative}, resp.) smoothing.  
To specify a direction of a smoothing of a crossing, we use a sufficient small segment on the crossing.  The small segment is called a \emph{marker} (see Fig.~\ref{kh2}).   
\begin{figure}[htbp]
\includegraphics[width=3cm]{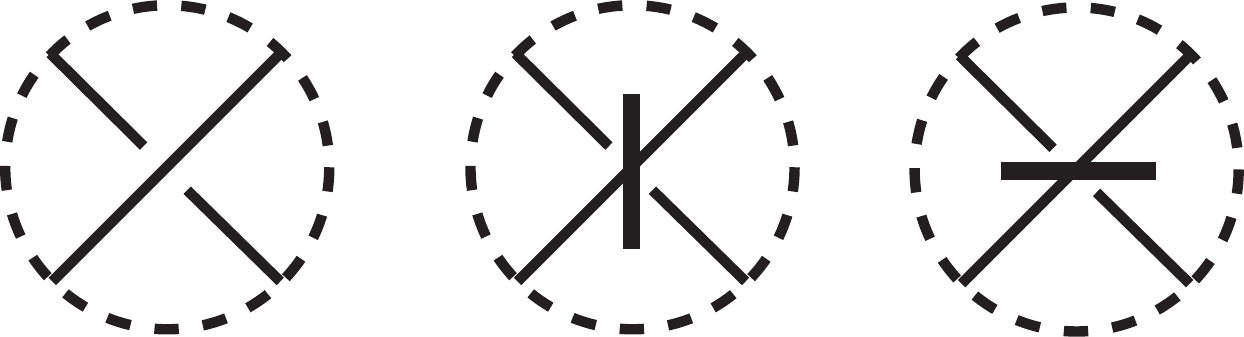}
\caption{(a) a disk ($\supset$ a single crossing), (b) positive marker, (c) negative marker.}\label{kh2}
\end{figure}
By definition, each signed marker determines the direction of a  smoothing for a crossing (Fig.~\ref{kh3}).
\begin{figure}[htbp]
\includegraphics[width=4cm]{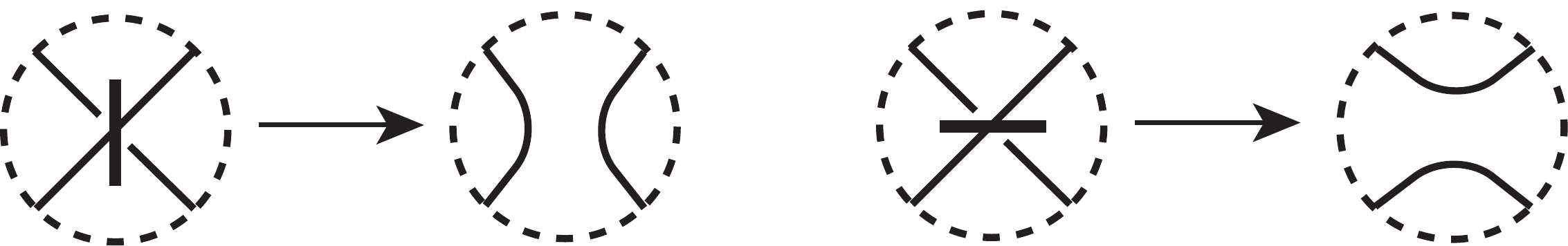}
\caption{Each signed marker determines the direction of a smoothing for a crossing.}\label{kh3}
\end{figure}    
For a given link diagram, suppose that we smooth every crossing along a marker.  Then, the smoothened link diagram is called a \emph{state}. 
 
Next, we assign a \emph{label}, $x$ or $1$, for every circle of a state.  For a state, we define a map $\de$ from the set of circles in the state to $\{ -1, 1 \}$ by $\de(x)= -1$ and $\de(1)=1$. 
A state whose circles have labels, each of which is $x$ or $1$, is called an \emph{enhanced state}.   Let $S$ be an enhanced state.  Then, let $\sigma(S)$ be the number of positive markers minus the number of negative markers for a state.  
Let $\tau(S)$ $=$ $\sum_{{\textrm{circle $y$ in $S$}}} \de(y)$.   
For a knot diagram $D$, let $w(D)$ be the number of positive crossings minus the number of negative crossings.  
Then, let 
\[
i(S) = \frac{1}{2}(w(D)-\sigma(S)), j(S) = w(D) + i(S) + \tau(S).  
\]
Here, note that the unnormalized Jones polynomial $\hat{J}(L)$ is defined by
\[
\hat{J}(L) = (-1)^{w(D)}\langle D \rangle.  
\]
Then, for a link diagram $D$ of a link $L$, $\hat{J}(L)$ is obtained by
\[
\hat{J}(L) = \sum_{{\text{enhanced state}}~S~{\text{of}}~D} (-1)^{i(S)} q^{j(S)}.  
\]
For the well-known Jones polynomial $V_L (t)$ with $V_{\text{unknot}} (t) =1$, the unnormalized Jones polynomial $\hat{J}(L)$ $=$ $(q+q^{-1}) V_L (q)$, with the variable $q$ replaced by $q= - t^{1/2}$.  

Now, we define the Khovanov chain group.  Let $D$ be a link diagram and $\mathcal{S}(D)$ the set of enhanced states of $D$.  
Let $C(D)$ $=$ $\mathbb{Z}_2 [\mathcal{S}(D)]$.  
Then, let 
$C^{i, j}(D)$ is the subgroup of $C(D)$ generated by enhanced states, each of which satisfies $i(S)=i$ and $j(S)=j$.
  
For a state $s$ having a positive marker on a crossing, there exists a state $s'$ such that $s'$ is obtained by replacing a single positive marker of $s$ with the negative marker at the crossing.  Let $S$ be an enhanced state that is a state $s$ with labels.  Then, let $T_S$ be an enhanced state that is the state $s'$ with labels, as shown in Fig.~\ref{kh4}.
\begin{figure}[htbp]
\centering
\includegraphics[width=5cm]{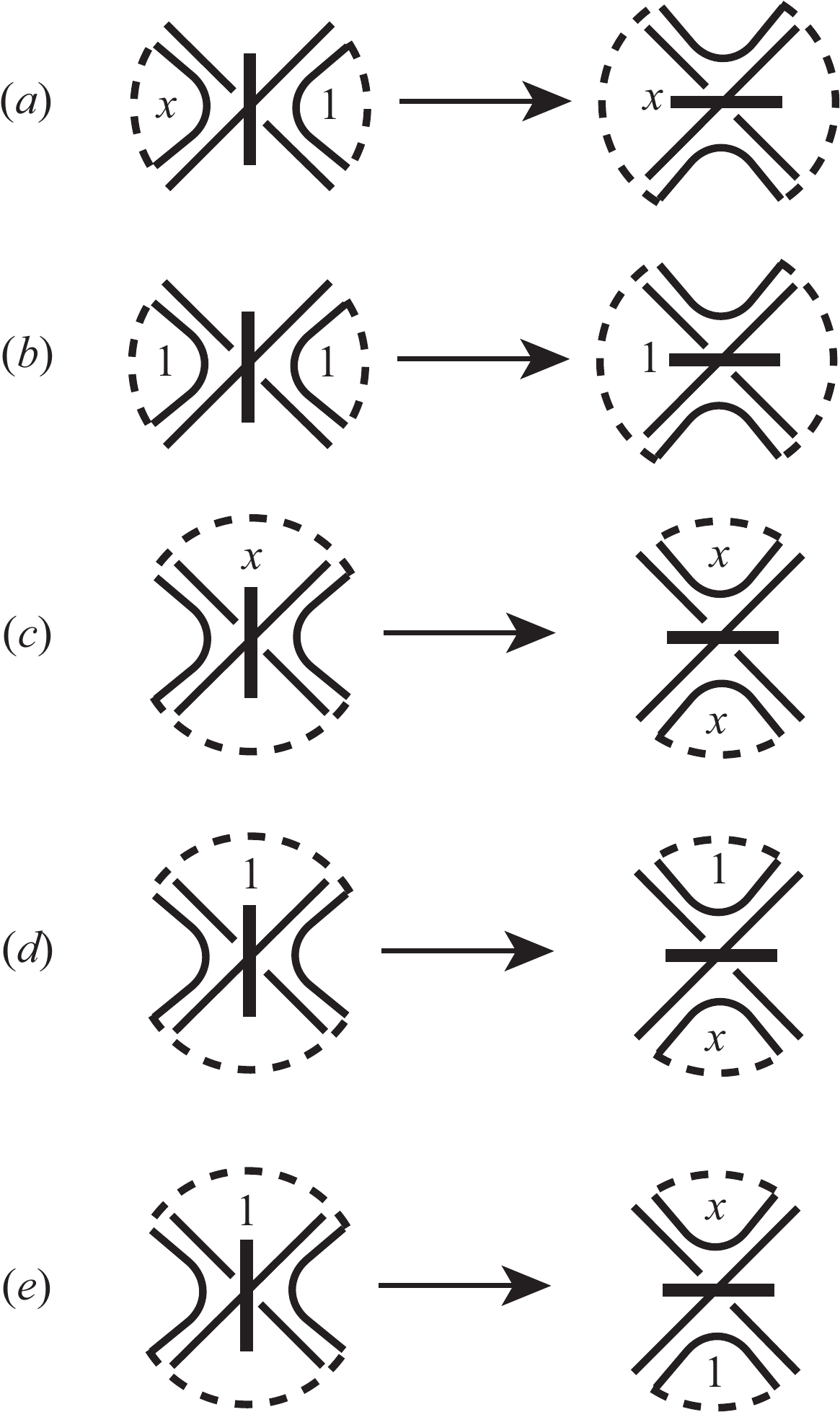}
\caption{Pairs $S, T_S$.}\label{kh4}
\end{figure}    
Then, for every enhanced state $S$, $d(S)$ is defined by
\[
d(S) = \sum_{{\text{enhanced state}}~S} ~T_S,
\]
By definition, this is naturally extended to the homomorphism $d$ from $C^{i, j}(D)$ to $C^{i+1, j}(D)$.  It is well-known fact that $d$ is a coboundary operator (i.e., $d^2 = 0$).  Traditionally, the homomorphism $d$ is called the differential in the case of Khovanov homology.  In \cite{khovanov}, Khovanov obtained Theorem~\ref{kh_thm}: 
\begin{theorem}[Khovanov \cite{khovanov}]
\label{kh_thm}
Let $L$ be a link and $D$ a link diagram.  The homology group $H^{i, j}(D)$ corresponding $\{C^{i, j}(D), d^i \}$ is an isotopy invariant of $L$, and thus, this homology can be denoted by $H^{i, j}(L)$.  The homology group $H^{i, j}(L)$ satisfies
\[
\hat{J}(L) = \sum_j q^j \sum_i (-1)^i \rank\, H^{i, j} (L).  
\] 
\end{theorem}

\subsection{A definition of the Khovanov homology by Manturov for virtual knots in the case of the coefficient $\mathbb{Z}_2$}\label{secManturov}
Manturov extended the definition of the Khovanov homology to that of virtual knots by adding the map between enhanced states obtained by virtual knot diagrams.  The problem is that the change of a single positive marker to define the differential does not require the change of the component enhanced states for all cases as in Fig.~\ref{kh4}.  This is because, for virtual knot diagrams, in general, there exists an enhanced state $S$ and a positive marker $p$ such that even if $p$ is changed, then the number of component of $S$ does not change.   
Fortunately, in the case of the coefficient $\mathbb{Z}_2$, the definition was extended to virtual knots straightforwardly by regarding these cases as zero maps (i.e., in particular, the switching of the markers unchanging the number of components change by $\pm 1$ corresponds to a zero map).  It is known that the extended homomorphism is also a coboundary operator. We denote the homomorphism by the same symbol $d$.  
In this paper, the Khovanov homology with the coefficient $\mathbb{Z}_2$ obtained by Manturov is denoted by $\kh^{i, j}$.   
In \cite{manturov}, Manturov obtained Theorem~\ref{ma_thm}: 
\begin{theorem}[Manturov \cite{manturov}]
\label{ma_thm}
Let $K$ be a virtual knot and $D$ a virtual knot diagram of $K$.  The homology group $\kh^{i, j}(D)$ corresponding $\{C^{i, j}(D), d^i \}$ is invariant under generalized Reidemeister moves, and thus, this homology can be denoted by $\kh^{i, j}(K)$.   
\end{theorem}

\section*{Acknowledgements}
The work was partially supported by Grant-in-Aid for Scientific Research~(S) (Number: 24224002) and by Grant for Basic Science Research Projects from The Sumitomo Foundation (Number: 160556).  
N.~Ito was a project researcher of Grant-in-Aid for Scientific Research~(S) (2016.4--2017.3).  The authors would like to thank the referee for the comments.


\begin{thebibliography}{99}
\bibitem{CKS} Carter, S.~Kamada, and M.~Saito, Stable equivalence of knots on surfaces and virtual knot cobordisms, Knots 2000, Korea, Vol.~1 (Yongpyong). \emph{J.~Knot Theory Ramifications} {\bf{11}} (2002), no.~3, 311--322. 
\bibitem{goussarov1} M.~Gusarov, On $n$-equivalence of knots and invariants of finite degree, \emph{Topology of manifolds and varieties}, 173--192, Adv.\ Soviet Math., 18, \emph{Amer.\ Math.\ Soc., Providence, RI,} 1994.  
\bibitem{goussarov} M.~Goussarov, Variations of knotted graphs.  The geometric technique of $n$-equivalence.  (Russian); \emph{translated form Algebra i Analiz} {\bf{12}} (2000), no.~4, 79--125 \emph{St.~Petersburg Math.\ J.\ } {\bf{12}} (2001), no.~4, 569--604.  
\bibitem{GPV} 
M.~Goussarov, M.~Polyak and O.~Viro, Finite type invariants of classical and
virtual knots, {\it Topology} {\bf 39} (2000), no.~5, 1045-1068.

\bibitem{habiro} K.~Habiro, Claspers and finite type invariants of links, \emph{Geom.\ Topol.\ } 4 (2000), 1--83.  
\bibitem{kanenobu1}
T.~Kanenobu, Examples on polynomial invariants of knots and links, \emph{Math.\
Ann.\ } {\bf{275}} (1986), no.~4, 555-572.

\bibitem{kanenobu2}
T.~Kanenobu, Forbidden moves unknot a virtual knot, {\it J.~Knot Theory Ramifications} {\bf 10} (2001), no.~1, 89-96.

\bibitem{Kauffman}
L.~H.~Kauffman, Virtual knot theory, {\it Europ.~J.~Combin.} {\bf 20} (1999), no.~7, 663-691.

\bibitem{khovanov}
M.~Khovanov, A categorification of the Jones polynomial. {\it Duke Math. J.} {\bf 101} (2000), no.~3, 359--426.

\bibitem{manturov}
V.~O.~Manturov, The Khovanov complex for virtual knots. (Russian. English, Russian summary), {\it translated from 
Fundam. Prikl. Mat.} {\bf 11} (2005), no.~4, 127--152,
{\it J. Math. Sci. (N.Y.)} {\bf 144} (2007), no.~5, 4451-4467.
 

\bibitem{nelson}
S.~Nelson, Unknotting virtual knots with Gauss diagram forbidden moves, {\it J.~Knot Theory Ramifications} {\bf 10} (2001), no.~6, 931-935.

\bibitem{Ohyama1} Y.~Ohyama, A new numerical invariant of knots induced from their regular diagrams, \emph{Topology Appl.} {\bf{37}} (1990), no.~3, 249--255.  

\bibitem{Ohyama2} Y.~Ohyama, Vassiliev invariants and similarity of knots, \emph{Proc.\ Amer.\ Math.\ Soc.} {\bf{123}} (1995), no.~1 287--291.  

\bibitem{Ohyama}
Y.~Ohyama, Vassiliev invariants and local moves of knots, {\url{http://www.math.kobe-u.ac.jp/publications/rlm15.pdf}} (2003).

\bibitem{PV}
M.~Polyak and O.~Viro, Gauss diagram formulas for Vassiliev invariants, \emph{Internat.\ Math.\ Res.\ Notices} {\bf{1994}}, 445ff., approx. 8pp.

\bibitem{Sakurai} M.~Sakurai, 2- and 3-variations and finite type invariants of degree 2 and 3, \emph{J. Knot Theory Ramifications} {\bf{22}} (2013), no.~8, 1350042, 20 pp.

\bibitem{Taniyama}
K.~Taniyama, On similarity of links, Gakujutu Kenkyu (issued by the school of education of Waseda University) {\bf{14}} (1992), 33--36.  

\bibitem{Va}
V.~A.~Vassiliev, Cohomology of knot spaces, {\it Theory of Singularities and its Applications}, Amer.~Math.~Soc. {\bf 20} (1990), 23-69, Adv.\ Soviet Math., 1, \emph{Amer.\ Soc., Providence, RI,} 1990.  

\bibitem{Vi} O.~Viro, Generic immersions of the circle to surfaces and the complex topology of real algebraic curves, \emph{Topology of real algebraic varieties and related topics,} 231--252, Amer.\ Math.\ Soc.\ Transl.\ Ser.~2, 173, Adv.\ Math.\ Sci., 29, \emph{Amer.\ Math.\ Soc., Providence, RI,} 1996.  
\end{thebibliography}
\end{document}